\documentclass[12pt,psamsfonts]{amsart}
\usepackage{amssymb,amsmath,amscd,eucal,epsfig}
\usepackage{tikz}

 \def\Dj{\hbox{D\kern-.73em\raise.30ex\hbox{-} \raise-.30ex\hbox{}}}
 \def\dj{\hbox{d\kern-.33em\raise.80ex\hbox{-} \raise-.80ex\hbox{\kern-.40em}}}

\def\<{\langle}                     
\def\>{\rangle}                     

\newtheorem{thm}{Theorem}[section]

\newtheorem{lem}[thm]{Lemma}
\newtheorem{cor}[thm]{Corollary}

\newcommand{\ben}{\begin{enumerate}}
\newcommand{\een}{\end{enumerate}}

\theoremstyle{plain}

\theoremstyle{definition}

\newtheorem{rem}{Remark}[section]

\numberwithin{equation}{section}
\begin{document}
	\title[Maximum Induced Matching Numbers of Grids]{On Maximum Induced Matching Numbers of Special Grids}
	
	\author[T.C Adefokun]{Tayo Charles Adefokun$^1$ }
	\address{$^1$Department of Computer and Mathematical Sciences,
		\newline \indent Crawford University,
		\newline \indent Nigeria}
	\email{tayoadefokun@crawforduniversity.edu.ng}
	
	\author[D.O. Ajayi]{Deborah Olayide Ajayi$^2$}
	\address{$^2$Department of Mathematics,
		\newline \indent University of Ibadan,
		\newline \indent Ibadan,
		\newline \indent Nigeria}
	\email{olayide.ajayi@mail.ui.edu.ng; adelaideajayi@yahoo.com}

	
	\begin{abstract} A subset $M$ of the edge set of a graph $G$ is an induced matching of $G$  if given any two $e_1,e_2 \in M$, none of the vertices on $e_1$ is adjacent to any of the vertices on $e_2$. Suppose that $MIM_G$, a positive integer, is the largest possible size of $M$ in $G$, then, $M$ is the maximum induced matching, $MIM$, of $G$ and $MIM_G$ is the maximum induced matching number of $G$. We obtain some upper bounds for the maximum induced matching numbers of some specific grids.
		
	\end{abstract}
		
\maketitle		
	
	\section{Introduction}
	
	For a graph $G$, let $V(G),E(G)$ be vertex and edge sets respectively and let $e \in E(G)$, we define $e=uv$, where $u,v \in V(G)$. Also, the respective order and size of $V(G)$ and $E(G)$ are $|V(G)|$ and $|E(G)|$. For some $M \subseteq E(G)$, $M$ is an induced matching of $G$ if for all $e_1=u_iu_j$ and $e_2=v_iv_j$ in $M$, $u_kv_l \notin M$, where $k$ and $l$ are from $\left\{i,j\right\}$. Induced matching, a variant of the matching problem, was introduced in 1982 by Stockmeyer and Vazirani \cite{SV1} and has also been studied under the names strong matchings, "risk free" marriage problem.  It has found theoretical and practical applications in a lot of areas including network problems and cryptology \cite{CST1}. For more on induced matching and its applications, see \cite{C1},\cite{CST1}, \cite{DD1}, \cite{E1} \cite{Z1}.
	
	The size of an induced matching is the number of edges in the induced matching and induced matching $M$ of $G$ with the largest possible size is known as the maximum induced matching which is denoted by $MIM$, its size, $MIM_G$, is called the maximum induced matching number (or the strong matching number) of $G$. Obtaining $MIM_G$  is $NP-$hard, even for regular bipartite graphs \cite{DD1}. However, $MIM_G$ of some graphs have been found in polynomial time (\cite{CST1}, \cite{GL1}).
	
 	A grid $G_{n,m}$ results from the Cartesian product of two paths $P_n$ and $P_m$, resulting in $n$-rows and $m$-colums. Marinescu-Ghemaci in \cite{RMG1}, obtained the $MIM_G$ for $G_{n,m}$, grid where both $n, m$ are even; either of $n$ and $m$ is even and a mumber of grids $G_{n,m}$ for which $nm$ is odd, which is called the odd grid \cite{AA1}. Marinescu-Ghemaci also gave useful lower and upper bounds and conjectured that the $MIM$ numbers of grids can be found in polynomial time. Furthermore, by combining the $MIM$ numbers of certain partitions of odd grids, it was shown that for any odd grid $G \equiv G_{n,m}$, $MIM_G \leq \left\lfloor \frac{nm+1}{4} \right\rfloor$. This bound was improved on in \cite{AA1} for the case where $n\geq 9$ and $m \equiv 1 \mod 4$.
	In this paper, the Marinescu-Ghemaci bound for the case where $n\geq 9$ and $m \equiv 3 \mod 4$ is considered and more compact values are obtained. The results in this work, combined with some of the results in \cite{RMG1}, confirm the $MIM$ numbers of certain graphs, whose $MIM$ numbers' lower bounds were established in \cite{RMG1}.

Section 2, of this work, is a review of definitions and preliminary results needed in this work, while in section 3, we obtain the maximum induced matching number of odd grids.

	\section{Definitions and Preliminary Results}
	Grid, $G_{n,m}$, as defined in this work, is the Cartesian product of paths $P_n$ and $P_m$ with $V(P_n)=\left\{u_1,u_2,\cdots, u_n\right\}$ and $V(P_m)=\left\{v_1,v_2,\cdots, v_m\right\} $. We adapt the following notations from \cite{AA1}: $V_i=\left\{u_1v_i, u_2v_i, \cdots, u_nv_i\right\} \subset V(G_{n,m}), i \in [1,m]$ and $U_i=\left\{u_iv_1, u_iv_2, \cdots, u_iv_m\right\}\subset V(G_{n,m}), i\in [1,n]$. For edge set $E(G_{n,m})$ of $G_{n,m}$, if $u_iv_j u_kv_j \in E(G_{n,m})$ and $u_iv_j u_iv_k \in E(G_{n,m})$, we write $u_{(^{i}_{k})}v_j \in E(G_{n,m})$ and $u_iv_{(^{j}_{k})} \in \in E(G_{n,m})$ respectively.
	
	A saturated vertex $v$ is any vertex on an edge in $M$, otherwise, $v$ is unsaturated. We define $v$ as saturable if it can be saturated relative to the nearest saturated vertex. Any vertex that is at least distant-$2$ from any saturated vertex is saturable.  The set of all saturated and saturable vertices on a graph $G$ is denoted by $V_{st}{G}$ and $V_{sb}(G)$ respectively. Clearly, $|V_{st}(G)|$ is even and $V_{st}(G) \subseteq V_{sb}(G)$. Free saturable vertices ($FSV$) are saturable vertices that can not be on any member of $M$, $FSV=V_{sb} \backslash V_{st}$. Let $G$ be a $G_{n,m}$ grid, we define $G^{|k|}$ as a $G_{n,k}$ subgrid of $G$ induced by $V_{i+1},V_{i+2},\cdots,V_{i+k}$.

The following results from \cite{RMG1} on $G$, a $G_{n,m}$ grid, are useful in this work:
\begin{lem}\label{lem2.1} Let $m,n \geq 2$ be two positive integers and let $G$ be a $G_{n,m}$ grid. Then,
\begin{enumerate}
\item
If $m \equiv 2 \mod 4$ and $n$ odd then $|V_{sb}(G)|=\frac{mn+2}{2}$; and $|V_{sb}(G)|=\frac{mn}{2}$  otherwise;
\item for $m\geq 3$, $m$ odd, $|V_{sb}(G)|=\frac{nm+1}{2}$, for $n\in \left\{3,5\right\}$.
\end{enumerate}
\end{lem}

\begin{thm}\label{thm2.2} Let $G$ be a $G_{n,m}$ grid with $2\leq n \leq m$. Then,
\begin{enumerate}
\item if $n$ even and $m$ even or odd, then $MIM_G=\left\lceil \frac{mn}{4}\right\rceil$;
\item if $n \in \left\{3,5\right\}$ then for
\begin{enumerate}
\item $ m \equiv 1 \mod 4$, $MIM_G=\frac{n(m-1)}{4}+1$
\item $m \equiv 3 \mod 4$, $MIM_G=\frac{n(m-1)+2}{4}$
\end{enumerate}
\end{enumerate}

\end{thm}	
The following theorem is the statement of the bound given by Marinescu-Ghemaci \cite{RMG1}.

\begin{thm}\label{thm2.4} Let $G$ be a $G_{n,m}$ grid, $m,n \geq 2$, $mn$ odd. Then $MIM_G \leq \left\lfloor \frac{mn+1}{4}\right\rfloor$.
\end{thm}

\section{Maximum induced matching number of odd grids}
The following result and the remark describe the importance of the saturation status of certain vertices in $G_{5,p}$ grid, where $p \equiv 2\mod 4$.

\begin{lem}\label{lema1} Let $G$ be a $G_{n,m}$ grid and let $\left\{V_i,V_{i+1},\cdots, V_{i+p}\right\} \subset G$ induce $G^{|p|}$, a $G_{5,p}$ subgrid of $G$, where $p \equiv 2\;mod\;4$. Suppose that $M_1$, is an induced matching of $G^{|p|}$ and that for $u_3v_i \in V(G^{|p|})$, $u_3v_i \notin V_{st}(G^{|p|})$. Then, $V_{st}(G^{|p|}) \leq 10k+4$ and $M_1$ is not an $MIM$ of $G^{|p|}$.
\end{lem}
\begin{proof} Let $p=4k+2$, $G^{|2|}$ and $G^{|p-2|}$ be partitions of $G_1$, induced by $\left\{V_i, V_{i+1}\right\}$ and $\left\{V_{i+2}, V_{i+3}, \cdots V_{i+p}\right\}$, respectively. Since $u_3v_i$ is not saturated in $G^{|2|}$, it easy to check that $|V_{sb}(G^{|2|})|=5$. From \cite{RMG1}, $|V_{sb}(G^{|p-2|})|=|V_{st}(G^{|p-2|})|=10k$. Thus $|V_{sb}(G^{|p|})| \leq |V_{sb}(G^{|2|})+V_{sb}(G^{|m-2|})| \leq 10k+5$ and therefore, $|V_{st}(G^{|p|})|=10k+4$ since $|V_{st}(G)|$ is even, for any graph $G$. This is a contradiction since by \cite{RMG1} $|V_{st}(G^{|p|})|=10k+6$.
\end{proof}

\begin{rem}\label{rema1} It should be noted that $M$ in Lemma \ref{lema1} will still not be $MIM$ of $G$ if for the vertex set  $A=\left\{u_1v_1,u_5v_1,u_1v_m,u_3v_m,u_5v_m\right\} \subset V(G)$, any member of $A$ is unsaturated.
\end{rem}

\begin{lem}\label{lema2} Suppose $u_{(^1_2)}v_i$,$u_5v_{\left(^{i-1}_i\right)} \in M$ or $u_{(^1_2)}v_i$,$u_5v_{\left(^i_{i+1}\right)} \in M$ where $M$ is an induced matching of $G$, a $G_{5,m}$ grid, $m \equiv 3 \mod 4$, $m \geq 23$ and $1<i<m$, $i \notin \left\{4, m-3\right\}$. Then $M$ is not a $MIM$ of G.
 \end{lem}
 \begin{proof} Let $G$ be partitioned into $G^{|m(1)|}$ and $G^{|m(2)|}$, which are respectively induced by $A=\left\{V_1,V_2, \cdots V_i\right\}$ and $B=\left\{V_{i+1}. V_{i+2}, \cdots V_{m}\right\}$. Suppose that $M$ is an $MIM$ of $G$.
 \\
Case 1: $i \equiv 1 \mod 4$. Let $m=4k+3$ and set $i=4t+1$, where $k \geq 5$ and $t > 0$. Then, $|m(1)|\equiv 1 \mod 4$ and $|m(2)|\equiv 2 \mod 4$. Since $u_1v_i, u_2v_i$, $u_5v_i$ and $u_5v_{i-1}$ are saturated vertices in $V_i$, then the only $FSV$ on $V_{i-1}$ is $u_3v_{i-1}$. Suppose that $u_3v_{i-1}$ remains unsaturated. Let $G^{|m(3)|} \subset G^{|m(1)|}$ induced by $\left\{V_1,V_2, \cdots V_{i-2}\right\}$, where $|m(3)|\equiv 3 \mod 4$. By \cite{RMG1}, $|V_{st}(G^{|m(3)|})|=10t-4$. Thus, $|V_{st}(G^{|m(1)|})| \leq 10t$. Suppose that $u_3v_{i-1}$ is saturated, then, $u_5v_{(^{i-2}_{i-1})} \in M$. Thus, $u_3v_{i-3} \in V_{i-3} \subset G^{|m(4)|}$, unsaturable, where $G^{|m(4)|}$ is $G^{|m(3)|} \backslash V_{i-2}$. Note that $|m(4)|\equiv 2 \mod 4$. From Lemma \ref{lema1}, therefore, $|V_{st}(G^{|m(4)|})| \leq 10t-6$ and thus, $|V_{st}G^{|m(1)|}| \leq 10t-6+6=10t$.
Now, since $u_1v_i, u_2v_i$ and $u_5v_i$ are saturated vertices in $V_i$, then, $u_3v_{i+1}, u_4v_{i+1} \in V(G^{|m(2)|})$ are saturable vertices in $G^{|m(2)|}$.

Claim: Edge $u_{(^3_4)}v_{i+1}$ belongs to $M$.

Reason: Suppose that both $u_3v_{i+1}$ and $u_4v_{i+1}$ are not saturated, then $V_{i+1}$ contains no saturable vertices. Let $G^{|m(2)|}\backslash \left\{V_{i+1}\right\}=G^{|m(5)|}$, where $|m(5)|\equiv 1 \mod 4.$ Thus, $|V_{st}(G)| \leq |V_{st}G^{|(m(1))|}|+|V_{st}(G^{|m(5)|})|=10k+2$. This implies that $M$ requires at least four more saturated vertices to be $MIM$ of $G$. However, $|V_{sb}(G^{|m(5)|})|=10(k-t)+3$ and suppose $u_3v_{i+1}, u_4v_{i+1} \in V_{st}(G)$, then $|V_{st}(G)| \leq 10k+5$, which in fact, is $|V_{st}(G)|=10k+4$. Thus if $u_{(^1_2)}v_i$,$u_5v_{\left(^{i-1}_i\right)} \in M$, then $M$ is not an $MIM$ of $G$.\\
Suppose that $u_{(^1_2)}v_i$,$u_5v_{\left(^i_{i+1}\right)} \in M$. Let $G^{|n(1)|}=G^{|m(1)|}\backslash \left\{V_i\right\}$ and $G^{|n(2)|}=G^{|m(2)|}+\left\{V_i\right\}$. Now, $|n(1)|\equiv 0 \mod 4$ and $|n(2)|\equiv 3 \mod 4$. We can see that $|V_{st}(G^{|n(2)|})|=10(k-t)+6$. Now, on $V_{i-1} \subset G^{|n(1)|}$, only vertices $u_3v_{i-1}$ and $u_4v_{i-1}$ are saturable. Suppose they are both not saturated after all. Let $G^{|n(3)|} \subset G^{|n(1)|}$ be induced by $\left\{V_1, V_2,\cdots, V_{i-2}\right\}$, where $|n(3)|\equiv 3 \mod 4$. $|V_{st}(G^{|n(3)|})|=10t-4$. Thus $|V_{st}(G)|=10k+2$. Therefore, $M$ requires four saturated vertices to be $MIM$ of $G$. Now, $|V_{sb}(G^{|n(3)|})|=10t-2$, and thus, $V(G^{|n(3)|})$ contains two extra $FSV$, say, $v_1,v_2$ which are not adjacent. Thus, the maximum number of saturable vertices from the vertex set $v_1, v_2, u_3v_{i-1},u_4v_{i-1}$ is $2$. Therefore, $|V_{st}(G)| \leq 10k+4$, which is a contradiction. \\
Case 2. For $i \equiv 2\mod 4$. Let $G^{|p(1)|}$ and $G^{|p(2)|}$ be partitions of $G$ induced by $\left\{V_1,V_2 \cdots V_i\right\}$ and $\left\{V_{i+1},V_{i+2}, \cdots V_m\right\}$, with $m=4k+3$ and $i=4t+2$. Let $u_{(^1_2)}v_i$ and $u_5v_{(^{i-1}_i)}\in M$. Since $u_{(^1_2)}v_i$ belongs $M$ of $G$, then $u_3v_i$ cannot be saturated. Thus, $|V_{st}(G^{|p(2)|})| \geq 10(k-t)+2$ for $M$ to be maximal. It can be seen that $|p(2)| \equiv 1 \mod 4$. Now, $u_3v_{i+1}$ and $u_4v_{i+1}$ are saturable vertices in $V_{i+1}$. Suppose both of them are not saturated, then for $G^{|p(3)|}$ induced by $\left\{V_{i+2}, V_{i+3}, \cdots V_m\right\}$, where $|p(3)| \equiv 0 \mod 4$, $|V_{st}(G^{|p(3)|})| \leq 10(k-t)$. Thus $u_3v_{i+1}$ and $v_4v_{i+1}$ are saturable vertices and in fact, $u_{(^3_4)}v_{i+1} \in M$. On $V_{i+2}$, therefore, there exists three saturable vertices $u_1v_{i+1}, u_2v_{i+2}$ and $u_5v_{i+5}$. Suppose none of these three vertices are saturated. Then, $|V_{st}(G^{|p(3)|})| \leq |V_{st}(G^{|p(4)|})|+2$, with $G^{|p(4)|}$ induced by $\left\{V_{i+3}, \cdots V_m\right\}$ and $|p(4)|\equiv 3 \mod 4$ and thus, $|V_{st}(G^{|p(2)|})| \leq 10(t-k)-2$. Therefore it requires extra four saturated vertices for $M$ to be maximal. There exist two other saturable vertices, $v_1, v_2 \in V(G^{|p(4)|})$ (since $V_{st}(G^{|p(4)|})=10(k-t)-4$ and $V_{sb}(G^{|p(4)|})=10(k-t)-2$). Clearly, $v_1,v_2$ are not adjacent, else they would have formed an edge in $M$. Suppose $v_1,v_2 \in V_{i+3}$. For $v_1$ and $v_2$ to be saturated, they have to be $u_5v_{i+3}$ and one of $u_1v_{i+3}$ and $u_2v_{i+3}$. Thus, $u_5v_{(^{i+2}_{i+3})}\in M$ and one of $u_1v_{(^{i+2}_{i+3})}$ $u_2v_{(^{i+2}_{i+3})}$ or $u_{(^1_2)}v_{i+2}$ belongs to $M$. Let $G^{|p(5)|}$ be induced by $\left\{V_{i+4}, \cdots V_m\right\}$, where $|p(5)|\equiv 2 \mod 4$. Now, since $v_5v_{(^{i+2}_{i+3})} \in M$, then $u_5v_{i+5} \in V_{i+4}$ is unsaturable and therefore, by Remark \ref{rema1}, $|V_{st}(G^{|p(5)|})|=10(k-t-1)+4$ and thus, $|V_{st}(G^{|p(2)|})|=10(k-t)$, which is less than required. \\
The case of $u_5v_{(^i_{i+1})} \in M$ is the same as the case of $u_5v_{(^{i-1}_i)} \in M$ for $i \equiv 2 \mod 4$.\\
Case 3: $i \equiv 0 \mod 4$, $i \geq 6$ or $i \leq m-5$, with $u_{(^1_2)}v_i, u_5v_{(^{i-1}_{i})} \in M$. Let $G^{|r(1)|}$ and $G^{|r(2)|}$ be partitions of $G$ which are induced respectively by $\left\{V_1,V_2, \cdots V_i\right\}$ and $\left\{V_{i+1},V_{i+2}, \cdots, V_m \right\}$. Since $i \equiv 0 \mod 4$, then $|r(1)|\equiv 0 \mod 4$, while $|r(2)|\equiv 3 \mod 4$. Also, $u_5v_{(^{i-1}_{i})} \in M$, implies $u_5v_{i-1}$ is unsaturable. Since $i-2 \equiv 2 \mod 4$, then by Lemma \ref{lema1} and Remark \ref{rema1}, $|V_{st}(G^{|r(1)|})| \leq 10t-2$, implying that for $M$ to be maximal, $|V_{st}(G^{|r(2)|})| \geq 10(k-t)+8$. It can be seen that $V_{i+1}$ has two only saturable vertices $u_3v_{i+1}$, $u_4v_{i+2}$ left. It should also be noted that if any of $u_3v_{i+1}$ and $u_4v_{i+2}$ is saturated, then $u_3v_{i+3}$ can not be saturated in $G^{|r(3)|}$, a subgrid of $G^{|r(2)|}$ induced by $\left\{V_{i+2}, V_{i+3}, \cdots V_m\right\}$, with $|r(3)|\equiv 2 \mod 4$. Thus suppose $u_3v_{i+1}, u_4v_{i+2} \in V_{st}(G)$, then $|V_{st}(G)| \leq 10(k-t)+4$. Likewise, if $u_3v_{i+1}, u_4v_{i+2} \notin V_{st}(G)$, $|V_{st}(G)| \leq 10t-2+10(k-t)+6$.

The case of $u_5v_{(^i_{i+1})} \in M$ follows the same argument as the case of $u_5v_{(^{i-1}_i)} \in M$.
 \end{proof}

{\tiny{
\begin{center}
\pgfdeclarelayer{nodelayer}
\pgfdeclarelayer{edgelayer}
\pgfsetlayers{nodelayer,edgelayer}
\begin{tikzpicture}
	\begin{pgfonlayer}{nodelayer}
	
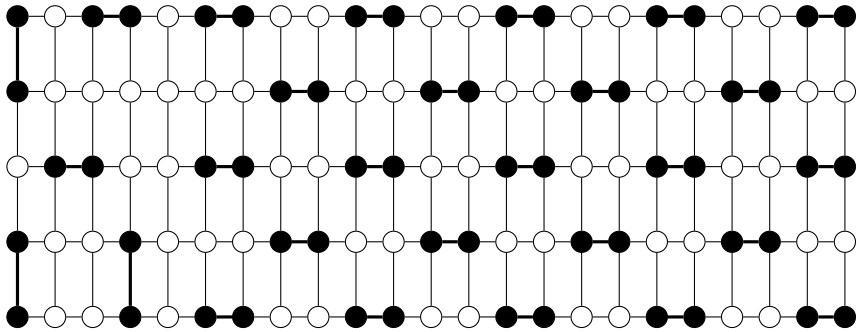
	\node [minimum size=0cm,]  at (-6.5,6.5) {Figure 1. A Grid $G \equiv G_{5,23}$ with $MIM_G=28$, $u_{^1_2}v_1,u_{^1_2}v_4 \in MIM$ of $G$};
	
	  \node [minimum size=0cm,draw,fill=black,circle] (1) at (-12,7) {};
		\node [minimum size=0cm,draw,circle] (2) at (-11.5,7) {};
		\node [minimum size=0cm,draw,circle] (3) at (-11,7) {};
		\node [minimum size=0cm,draw,fill=black,circle] (4) at (-10.5,7) {};
		\node [minimum size=0cm,draw,circle] (5) at (-10,7) {};
		\node [minimum size=0cm,draw,fill=black,circle] (6) at (-9.5,7) {};
		\node [minimum size=0cm,draw,fill=black,circle] (7) at (-9,7) {};
		\node [minimum size=0cm,draw,circle] (8) at (-8.5,7) {};
		\node [minimum size=0cm,draw,circle] (9) at (-8,7) {};
		\node [minimum size=0cm,draw,fill=black,circle] (10) at (-7.5,7) {};
    \node [minimum size=0cm,draw,fill=black,circle] (11) at (-7,7) {};
		\node [minimum size=0cm,draw,circle] (12) at (-6.5,7) {};
		\node [minimum size=0cm,draw,circle] (13) at (-6,7) {};
		\node [minimum size=0cm,draw,fill=black,circle] (14) at (-5.5,7) {};
		\node [minimum size=0cm,draw,fill=black,circle] (15) at (-5,7) {};
		\node [minimum size=0cm,draw,circle] (16) at (-4.5,7) {};
		\node [minimum size=0cm,draw,circle] (17) at (-4,7) {};
		\node [minimum size=0cm,draw,fill=black,circle] (18) at (-3.5,7) {};
		\node [minimum size=0cm,draw,fill=black,circle] (19) at (-3,7) {};
		\node [minimum size=0cm,draw,circle] (20) at (-2.5,7) {};
		\node [minimum size=0cm,draw,circle] (21) at (-2,7) {};
		\node [minimum size=0cm,draw,fill=black,circle] (22) at (-1.5,7) {};
		\node [minimum size=0cm,draw,fill=black,circle] (23) at (-1,7) {};

		\node [minimum size=0cm,draw,fill=black, circle] (24) at (-12,8) {};
		\node [minimum size=0cm,draw,circle] (25) at (-11.5,8) {};
		\node [minimum size=0cm,draw,circle] (26) at (-11,8) {};
		\node [minimum size=0cm,draw,fill=black,circle] (27) at (-10.5,8) {};
		\node [minimum size=0cm,draw,circle] (28) at (-10,8) {};
		\node [minimum size=0cm,draw,circle] (29) at (-9.5,8) {};
		\node [minimum size=0cm,draw,circle] (30) at (-9,8) {};
		\node [minimum size=0cm,draw,fill=black,circle] (31) at (-8.5,8) {};
		\node [minimum size=0cm,draw,fill=black,circle] (32) at (-8,8) {};
		\node [minimum size=0cm,draw,circle] (33) at (-7.5,8) {};
		\node [minimum size=0cm,draw,circle] (34) at (-7,8) {};
		\node [minimum size=0cm,draw,fill=black,circle] (35) at (-6.5,8) {};
		\node [minimum size=0cm,draw,fill=black,circle] (36) at (-6,8) {};
		\node [minimum size=0cm,draw,circle] (37) at (-5.5,8) {};
		\node [minimum size=0cm,draw,circle] (38) at (-5,8) {};
		\node [minimum size=0cm,draw,fill=black,circle] (39) at (-4.5,8) {};
		\node [minimum size=0cm,draw,fill=black,circle] (40) at (-4,8) {};
		\node [minimum size=0cm,draw,circle] (41) at (-3.5,8) {};
		\node [minimum size=0cm,draw,circle] (42) at (-3,8) {};
		\node [minimum size=0cm,draw,fill=black,circle] (43) at (-2.5,8) {};
		\node [minimum size=0cm,draw,fill=black,circle] (44) at (-2,8) {};
		\node [minimum size=0cm,draw,circle] (45) at (-1.5,8) {};
		\node [minimum size=0cm,draw,circle] (46) at (-1,8) {};

		\node [minimum size=0cm,draw,circle] (47) at (-12,9) {};
		\node [minimum size=0cm,draw,fill=black,circle] (48) at (-11.5,9) {};
		\node [minimum size=0cm,draw,fill=black,circle] (49) at (-11,9) {};
		\node [minimum size=0cm,draw,circle] (50) at (-10.5,9) {};
		\node [minimum size=0cm,draw,circle] (51) at (-10,9) {};
		\node [minimum size=0cm,draw,fill=black,circle] (52) at (-9.5,9) {};
		\node [minimum size=0cm,draw,fill=black,circle] (53) at (-9,9) {};
		\node [minimum size=0cm,draw,circle] (54) at (-8.5,9) {};
		\node [minimum size=0cm,draw,circle] (55) at (-8,9) {};
		\node [minimum size=0cm,draw,fill=black,circle] (56) at (-7.5,9) {};

		\node [minimum size=0cm,draw,fill=black,circle] (57) at (-7,9) {};
		\node [minimum size=0cm,draw,circle] (58) at (-6.5,9) {};
		\node [minimum size=0cm,draw,circle] (59) at (-6,9) {};
		\node [minimum size=0cm,draw,fill=black,circle] (60) at (-5.5,9) {};
		\node [minimum size=0cm,draw,fill=black,circle] (61) at (-5,9) {};
		\node [minimum size=0cm,draw,circle] (62) at (-4.5,9) {};
		\node [minimum size=0cm,draw,circle] (63) at (-4,9) {};
		\node [minimum size=0cm,draw,fill=black,circle] (64) at (-3.5,9) {};
		\node [minimum size=0cm,draw,fill=black,circle] (65) at (-3,9) {};
		\node [minimum size=0cm,draw,circle] (66) at (-2.5,9) {};
		\node [minimum size=0cm,draw,circle] (67) at (-2,9) {};
		\node [minimum size=0cm,draw,fill=black,circle] (68) at (-1.5,9) {};
		\node [minimum size=0cm,draw,fill=black,circle] (69) at (-1,9) {};

		\node [minimum size=0cm,draw,fill=black,circle] (70) at (-12,10) {};
		\node [minimum size=0cm,draw,circle] (71) at (-11.5,10) {};
		\node [minimum size=0cm,draw,circle] (72) at (-11,10) {};
		\node [minimum size=0cm,draw,circle] (73) at (-10.5,10) {};
		\node [minimum size=0cm,draw,circle] (74) at (-10,10) {};
			\node [minimum size=0cm,draw,circle] (75) at (-9.5,10) {};
		\node [minimum size=0cm,draw,circle] (76) at (-9,10) {};
		\node [minimum size=0cm,draw,fill=black,circle] (77) at (-8.5,10) {};
		\node [minimum size=0cm,draw,fill=black,circle] (78) at (-8,10) {};
		\node [minimum size=0cm,draw,circle] (79) at (-7.5,10) {};

		\node [minimum size=0cm,draw,circle] (80) at (-7,10) {};
		\node [minimum size=0cm,draw,fill=black,circle] (81) at (-6.5,10) {};
		\node [minimum size=0cm,draw,fill=black,circle] (82) at (-6,10) {};
		\node [minimum size=0cm,draw,circle] (83) at (-5.5,10) {};
		\node [minimum size=0cm,draw,circle] (84) at (-5,10) {};
		\node [minimum size=0cm,draw,fill=black,circle] (85) at (-4.5,10) {};
		\node [minimum size=0cm,draw,fill=black,circle] (86) at (-4,10) {};
		\node [minimum size=0cm,draw,circle] (87) at (-3.5,10) {};
		\node [minimum size=0cm,draw,circle] (88) at (-3,10) {};
		\node [minimum size=0cm,draw,fill=black,circle] (89) at (-2.5,10) {};
		\node [minimum size=0cm,draw,fill=black,circle] (90) at (-2,10) {};
		\node [minimum size=0cm,draw,,circle] (91) at (-1.5,10) {};
		\node [minimum size=0cm,draw,circle] (92) at (-1,10) {};

		\node [minimum size=0cm,draw,fill=black,circle] (93) at (-12,11) {};
		\node [minimum size=0cm,draw,circle] (94) at (-11.5,11) {};
		\node [minimum size=0cm,draw,fill=black,circle] (95) at (-11,11) {};
		\node [minimum size=0cm,draw,fill=black,circle] (96) at (-10.5,11) {};
		\node [minimum size=0cm,draw,circle] (97) at (-10,11) {};
		\node [minimum size=0cm,draw,fill=black,circle] (98) at (-9.5,11) {};
		\node [minimum size=0cm,draw,fill=black,circle] (99) at (-9,11) {};
		\node [minimum size=0cm,draw,circle] (100) at (-8.5,11) {};
		\node [minimum size=0cm,draw,circle] (101) at (-8,11) {};
		\node [minimum size=0cm,draw,fill=black,circle] (102) at (-7.5,11) {};
    \node [minimum size=0cm,draw,fill=black!,circle] (103) at (-7,11) {};
		\node [minimum size=0cm,draw,circle] (104) at (-6.5,11) {};
		\node [minimum size=0cm,draw,circle] (105) at (-6,11) {};
		\node [minimum size=0cm,draw,fill=black,circle] (106) at (-5.5,11) {};
		\node [minimum size=0cm,draw,fill=black,circle] (107) at (-5,11) {};
		\node [minimum size=0cm,draw,circle] (108) at (-4.5,11) {};
		\node [minimum size=0cm,draw,circle] (109) at (-4,11) {};
		\node [minimum size=0cm,draw,fill=black,circle] (110) at (-3.5,11) {};
		\node [minimum size=0cm,draw,fill=black,circle] (111) at (-3,11) {};
		\node [minimum size=0cm,draw,circle] (112) at (-2.5,11) {};
		\node [minimum size=0cm,draw,circle] (113) at (-2,11) {};
		\node [minimum size=0cm,draw,fill=black,circle] (114) at (-1.5,11) {};
		\node [minimum size=0cm,draw,fill=black,circle] (115) at (-1,11) {};

			\end{pgfonlayer}
				\begin{pgfonlayer}{edgelayer}
		\draw [thin=1.00] (1) to (2);
		\draw [thin=1.00] (2) to (3);
		\draw [thin=1.00] (3) to (4);
		\draw [thin=1.00] (4) to (5);
		\draw [thin=1.00] (5) to (6);
		\draw [very thick=1.00] (6) to (7);
		\draw [thin=1.00] (7) to (8);
		\draw [thin=1.00] (8) to (9);
			\draw [thin=1.00] (9) to (10);
		\draw [very thick=1.00] (10) to (11);
		\draw [thin=1.00] (11) to (12);
		\draw [thin=1.00] (12) to (13);
		\draw [thin=1.00] (13) to (14);
		\draw [very thick=1.00] (14) to (15);
		\draw [thin=1.00] (15) to (16);
		\draw [thin=1.00] (16) to (17);
		\draw [thin=1.00] (17) to (18);
		\draw [very thick=1.00] (18) to (19);
		\draw [thin=1.00] (19) to (20);
		\draw [thin=1.00] (20) to (21);
		\draw [thin=1.00] (21) to (22);
		\draw [very thick=1.00] (22) to (23);

			\draw [thin=1.00] (24) to (25);
		\draw [thin=1.00] (25) to (26);
		\draw [thin=1.00] (26) to (27);
		\draw [thin=1.00] (27) to (28);
		\draw [thin=1.00] (28) to (29);
		\draw [thin=1.00] (29) to (30);
		\draw [thin=1.00] (30) to (31);
		\draw [very thick=1.00] (31) to (32);
		\draw [thin=1.00] (32) to (33);
			\draw [thin=1.00] (33) to (34);
		\draw [thin=1.00] (34) to (35);
		\draw [very thick=1.00] (35) to (36);
		\draw [thin=1.00] (36) to (37);
		\draw [thin=1.00] (37) to (38);
		\draw [thin=1.00] (38) to (39);
		\draw [very thick=1.00] (39) to (40);
		\draw [thin=1.00] (40) to (41);
		\draw [thin=1.00] (41) to (42);
		\draw [thin=1.00] (42) to (43);
		\draw [very thick=1.00] (43) to (44);
		\draw [thin=1.00] (44) to (45);
		\draw [thin=1.00] (45) to (46);
		
			\draw [thin=1.00] (47) to (48);
		\draw [very thick=1.00] (48) to (49);
		\draw [thin=1.00] (49) to (50);
		\draw [thin=1.00] (50) to (51);
		\draw [thin=1.00] (51) to (52);
		\draw [very thick=1.00] (52) to (53);
		\draw [thin=1.00] (53) to (54);
		\draw [thin=1.00] (54) to (55);
		\draw [thin=1.00] (55) to (56);
			\draw [very thick=1.00] (56) to (57);
		\draw [thin=1.00] (57) to (58);
		\draw [thin=1.00] (58) to (59);
		\draw [thin=1.00] (59) to (60);
		\draw [very thick=1.00] (60) to (61);
		\draw [thin=1.00] (61) to (62);
		\draw [thin=1.00] (62) to (63);
		\draw [thin=1.00] (63) to (64);
		\draw [very thick=1.00] (64) to (65);
		\draw [thin=1.00] (65) to (66);
		\draw [thin=1.00] (66) to (67);
		\draw [thin=1.00] (67) to (68);
		\draw [very thick=1.00] (68) to (69);

   	\draw [thin=1.00] (70) to (71);
		\draw [thin=1.00] (71) to (72);
		\draw [thin=1.00] (72) to (73);
		\draw [thin=1.00] (73) to (74);
		\draw [thin=1.00] (74) to (75);
		\draw [thin=1.00] (75) to (76);
		\draw [thin=1.00] (76) to (77);
		\draw [very thick=1.00] (77) to (78);
		\draw [thin=1.00] (78) to (79);
		\draw [thin=1.00] (79) to (80);
		\draw [thin=1.00] (80) to (81);
		\draw [very thick=1.00] (81) to (82);
		\draw [thin=1.00] (82) to (83);
		\draw [thin=1.00] (83) to (84);
		\draw [thin=1.00] (84) to (85);
		\draw [very thick=1.00] (85) to (86);
		\draw [thin=1.00] (86) to (87);
		\draw [thin=1.00] (87) to (88);
		\draw [thin=1.00] (88) to (89);
		\draw [very thick=1.00] (89) to (90);
		\draw [thin=1.00] (90) to (91);
		\draw [thin=1.00] (91) to (92);

			\draw [thin=1.00] (93) to (94);
		\draw [thin=1.00] (94) to (95);
		\draw [very thick=1.00] (95) to (96);
		\draw [thin=1.00] (96) to (97);
		\draw [thin=1.00] (97) to (98);
		\draw [very thick=1.00] (98) to (99);
		\draw [thin=1.00] (99) to (100);
		\draw [thin=1.00] (100) to (101);
		\draw [thin=1.00] (101) to (102);
		\draw [very thick=1.00] (102) to (103);
		\draw [thin=1.00] (103) to (104);
		\draw [thin=1.00] (104) to (105);
		\draw [thin=1.00] (105) to (106);
		\draw [very thick=1.00] (106) to (107);
		\draw [thin=1.00] (107) to (108);
		\draw [thin=1.00] (108) to (109);
		\draw [thin=1.00] (109) to (110);
		\draw [very thick=1.00] (110) to (111);
		\draw [thin=1.00] (111) to (112);
		\draw [thin=1.00] (112) to (113);
		\draw [thin=1.00] (113) to (114);
		\draw [very thick=1.00] (114) to (115);

		\draw [very thick=1.00] (1) to (24);
		\draw [thin=1.00] (24) to (47);
		\draw [thin=1.00] (47) to (70);
		\draw [very thick=1.00] (70) to (93);
		
		\draw [thin=1.00] (2) to (25);
		\draw [thin=1.00] (25) to (48);
		\draw [thin=1.00] (48) to (71);
		\draw [thin=1.00] (71) to (94);
		
		\draw [thin=1.00] (3) to (26);
		\draw [thin=1.00] (26) to (49);
		\draw [thin=1.00] (49) to (72);
		\draw [thin=1.00] (72) to (95);
		
		\draw [very thick=1.00] (4) to (27);
		\draw [thin=1.00] (27) to (50);
		\draw [thin=1.00] (50) to (73);
		\draw [thin=1.00] (73) to (96);
		
		\draw [thin=1.00] (5) to (28);
		\draw [thin=1.00] (28) to (51);
		\draw [thin=1.00] (51) to (74);
		\draw [thin=1.00] (74) to (97);
		
		\draw [thin=1.00] (6) to (29);
		\draw [thin=1.00] (29) to (52);
		\draw [thin=1.00] (52) to (75);
		\draw [thin=1.00] (75) to (98);
		
		\draw [thin=1.00] (7) to (30);
		\draw [thin=1.00] (30) to (53);
		\draw [thin=1.00] (53) to (76);
		\draw [thin=1.00] (76) to (99);
		
		\draw [thin=1.00] (8) to (31);
		\draw [thin=1.00] (31) to (54);
		\draw [thin=1.00] (54) to (77);
		\draw [thin=1.00] (77) to (100);
		
			\draw [thin=1.00] (9) to (32);
		\draw [thin=1.00] (32) to (55);
		\draw [thin=1.00] (55) to (78);
		\draw [thin=1.00] (78) to (101);
		
		\draw [thin=1.00] (10) to (33);
		\draw [thin=1.00] (33) to (56);
		\draw [thin=1.00] (56) to (79);
		\draw [thin=1.00] (79) to (102);
		
		\draw [thin=1.00] (11) to (34);
		\draw [thin=1.00] (34) to (57);
		\draw [thin=1.00] (57) to (80);
		\draw [thin=1.00] (80) to (103);

		\draw [thin=1.00] (12) to (35);
		\draw [thin=1.00] (35) to (58);
		\draw [thin=1.00] (58) to (81);
		\draw [thin=1.00] (81) to (104);

		\draw [thin=1.00] (13) to (36);
		\draw [thin=1.00] (36) to (59);
		\draw [thin=1.00] (59) to (82);
		\draw [thin=1.00] (82) to (105);
		
		\draw [thin=1.00] (14) to (37);
		\draw [thin=1.00] (37) to (60);
		\draw [thin=1.00] (60) to (83);
		\draw [thin=1.00] (83) to (106);
		
		\draw [thin=1.00] (15) to (38);
		\draw [thin=1.00] (38) to (61);
		\draw [thin=1.00] (61) to (84);
		\draw [thin=1.00] (84) to (107);
		
		\draw [thin=1.00] (16) to (39);
		\draw [thin=1.00] (39) to (62);
		\draw [thin=1.00] (62) to (85);
		\draw [thin=1.00] (85) to (108);
		
	  \draw [thin=1.00] (17) to (40);
		\draw [thin=1.00] (40) to (63);
		\draw [thin=1.00] (63) to (86);
		\draw [thin=1.00] (86) to (109);
		
		\draw [thin=1.00] (18) to (41);
		\draw [thin=1.00] (41) to (64);
		\draw [thin=1.00] (64) to (87);
		\draw [thin=1.00] (87) to (110);
		
		\draw [thin=1.00] (19) to (42);
		\draw [thin=1.00] (42) to (65);
		\draw [thin=1.00] (65) to (88);
		\draw [thin=1.00] (88) to (111);
		
		\draw [thin=1.00] (20) to (43);
		\draw [thin=1.00] (43) to (66);
		\draw [thin=1.00] (66) to (89);
		\draw [thin=1.00] (89) to (112);
		
		\draw [thin=1.00] (21) to (44);
		\draw [thin=1.00] (44) to (67);
		\draw [thin=1.00] (67) to (90);
		\draw [thin=1.00] (90) to (113);
		
		\draw [thin=1.00] (22) to (45);
		\draw [thin=1.00] (45) to (68);
		\draw [thin=1.00] (68) to (91);
		\draw [thin=1.00] (91) to (114);
		
		\draw [thin=1.00] (23) to (46);
		\draw [thin=1.00] (46) to (69);
		\draw [thin=1.00] (69) to (92);
		\draw [thin=1.00] (92) to (115);
		
	\end{pgfonlayer}
\end{tikzpicture}
\end{center}
}}

\begin{rem}\label{rema1a}
 \ben \item In the case of $i \equiv 0 \mod m$ in \ref{lema2}, $M$ remains a $MIM$ when $i=4$ or when $i=m-3$ as seen in Figure 1 of $|MIM|=28$ of $G_{5,23}$.
%
\item \label{rema2} It should be noted that the case of $i \equiv 3 \mod 4$ has been taken care of by the case of $i \equiv 1 \mod 4$ by 'flipping' the grid from right to left or vice versa.
\item    \label{rema3} From Lemma \ref{lema2}, we see that if for some induced matching $M$ of $G_{5,m}$, $m \equiv 3 \mod 4$, $u_{(^1_2)}v_i$ and $u_5v_{(^{i-1}_{i})}$ (or $u_5v_{(^i_{i+2})}) \in M$, then $M$ is not a maximal induced matching of $G$ for any $1 < i < m$.
    \een
 \end{rem}


 Next we investigate some $M$ of $G_{5,m}$ if it contains $u_{(^1_2)}v_i$ and $u_{^4_5}v_i$.

\begin{lem}\label{lema3} Suppose $G=G_{5,m}$, where $m \geq 23$ and $m \equiv 3 \mod 4$. Let $u_{(^1_2)}v_i, u_{(^4_5)}v_i  \in M$, an induced matching of $G$ and $1<i<m$, $i \not\equiv 0 \mod 4$ then $M$ is not a $MIM$ of $G$.
\end{lem}
  \begin{proof} Suppose that $i \equiv 2 \mod 4$. Let $G^{|m(1)|}$ and $G^{|m(2)|}$ be partitions of $G$ induced by $\left\{V_1, V_2, \cdots, V_i\right\}$ and $\left\{V_{i+1}, V_{i+2}, \cdots, V_m\right\}$. Since $u_{(^1_2)}v_i, u_{(^4_5)}v_1 \in M$, then, $u_3v_i$ is unsaturated. Let $i=4t+2$, for some positive integer $t$, by Lemma \ref{lema2} then $|V_{st}(G^{|m(1)|})|=10t+4$. Now, only $u_3v_{i+1}$ is saturable on $V_{i+1}$. Let $G^{|m(3)|} \subset G^{|m(2)|}$, induced by $\left\{V_{i+2}, \cdots V_m\right\}$. Clearly $|m(3)|=|m(2)|-1=4(k-t)$. Therefore, $|V_{st}(G^{|m(3)|}+u_3v_i)| \leq 10(k-t)+1$, which in fact is $10(k-t)$. Thus, $|V_{st}(G)|=10k+4$.

Now, suppose $i \equiv 1 \mod 4$. Let $G^{|n(1)|}$ be induced by $\left\{V_1,V_2, \cdots, V_i\right\}$ and let $G^{|n(2)|}$ be induced by $\left\{V_{i+1}V_{i+2},\cdots,V_m\right\}$. Since $|n(1)|=4t+1$, it is easy to see that $|n(2)| \equiv 2 \mod 4$ and hence, $|n(2)|=4(k-t)+2$. \\
Claim. For $M$ to be maximal, both $u_3v_{i-1}$ and $u_3v_{i+1}$ must be saturated. \\
Reason: Suppose, say $u_3v_{i-1}$ is not saturated. Then, no vertex on $V_{i-1}$ is saturable. Now, let $\left\{V_1, V_2,  \cdots V_{i-2}\right\}$ induce grid $G^{|n(3)|}$, with $|n(3)|\equiv 3 \mod 4$. $|V_{st}(G^{|n(3)|})|=10t-4$, and thus, $G^{|n(1)|}=10t$ Also, Let $G^{|n(4)|}$ be induced by  $\left\{V_{i+2},V_{i+3, \cdots ,V_m}\right\}$. Since $|n(4)|=4(k-t)+1$, then for $G^{|n(4)|}+u_5v_{i+1}$, $|V_{sb}[(G^{|n(4)|})+u_3v_{i+1}]|=10(k-t)+4$. Therefore, $|V_{st}(G)| \leq 10k+4$. Now Suppose $u_3v_{(^{i-2}_{i-1})} \in M$. Then, given $G^{|n(5)|}$, induced by $\left\{V_1,V_2, \cdots, V_{i-3}\right\}$. We can see that $|n(5)|\equiv 2 \mod 4$. By Lemma \ref{lema1}, $|V_{st}(G^{|n(5)|})|=10t-6$. Thus, $|V_{st}(G^{|n(1)|})|=10t$ and therefore, $|V_{st}(G)| \leq 10k+4$.
  \end{proof}
\begin{rem} Like in Remark \ref{rema1a}, for $i \equiv 0 \mod 4$, it can be seen that $u_{(^1_2)}v_1, u_{(^1_2)}v_4$ or $u_{(^1_2)}v_{m-3}, u_{(^1_2)}v_m$ can be in $M$ if $M$ is $MIM$ of $G$. Also given $i \equiv 0 \mod 4$ and $4 < i < m-3 $, for at most only one $i$, from $1$ to $m$, $u_{(^1_2)}v_i$ can be a member of maximal $M$.

\end{rem}

Next we investigate the maximality of the induced matching of $G=G_{5,m}$, $m \equiv 3 \mod 4$.

\begin{lem}\label{lema4}
Let $u_{(^1_2)}v_i$,$u_4v_{\left(^{i-1}_i\right)} \in M$ or $u_{(^1_2)}v_i$,$u_4v_{\left(^i_{i+1}\right)} \in M$ where $M$ is an induced matching of $G$, a $G_{5,m}$ grid, $m \equiv 3 \mod 4$, $m \geq 23$ and $1<i<m$, $i \not\equiv 0 \mod 4$. Then $M$ is not a $MIM$ of G.
\end{lem}
\begin{proof} Case 1: Let $i \equiv 1 \mod 4$. Suppose that $m=4k+3$ and $i=4t+1$,  $t \geq 1$.  Let $G^{|m(1)|}$ and $G^{|m(2)|}$ be two partitions of $G$, induced by $\left\{V_1,V_2, \cdots, V_i\right\}$ and $\left\{V_{i+1},V_{i+2}, \cdots V_m\right\}$ respectively. Since $u_{(^1_2)}v_i, u_4v_{(^{i-1}_i)} \in M$, then there is no other saturated vertex on both of $V_{i-1}$ and $V_i$. Let $G^{|m(3)|} \subset G^{|m(1)|}$ be a grid induced by $\left\{V_1,V_2,\cdots, V_{i-2}\right\}$. Now, $n(3)\equiv 3 \mod 4$. Therefore, $|V_{st}(G^{|m(3)|})|=10t-4$ and hence, $|V_{st}(G^{|{m(1)|}})|=10t$. Now, $|m(2)|\equiv 2 \mod 4$, since $u_{(^1_2)}v_i \in M$, then $u_1v_{i+1} \in V_{i+1}$ is unsaturable. From a previous result, $|V_{st}(G^{|n(2)|})|=10(k-t)+4$ and thus, $|V_{st}(G)|=10k+4$. For $u_4v_{(^i_{i+1})} \in M$. Let $G^{|n(1)|}$ and $G^{|n(2)|}$ be induced by $G^{|m(1)|}\backslash V_i$ and $G^{|m(2)|}+V_i$. Then, $|n(1)|\equiv 0 \mod 4$ and $|n(2)|=4(k-t)+3$. It can be seen that on $V_{i-1}$, only $u_3v_{i-1}$ and $u_5v_{i-1}$ are saturable vertices. \\
Claim: Vertices $u_3v_{i-1}$ and $u_5v_{i-1}$ are not saturable for $M$ to be maximal. \\
Reason: Suppose without loss of generality, that any of $u_3v_{i-1}$ and $u_5v_{i-1}$ is saturated, say $u_5v_{i-1}$. Then $u_5v_{(^{i-2}_{i-1}) \in M}$. This implies that $v_5v_{i-3}$ is not saturable in $V_{i-3}$. Now $\left\{V_1, V_2, \cdots, V_{i-3}\right\}$ induces a grid $G^{(|n(4)|)}$ and $|n(4)|\equiv 2 \mod 4$. Then, $|V_{st}(G^{|m(4)|})|=10t-6$ and thus, $|V_{st}(G^{|n(1)|})|=10t-4$. Now, since $|n(2)|=4(k-t)+3$, $|V_{st}(G^{|m(2)|})|=10(k-t)+6$ and therefore, $|V_{st}(G)|=10k+2$.\\
Case 2: For $i \equiv 2 \mod 4$. Let $G^{|n(1)|}$ and $G^{|n(2)|}$ be two partitions of $G$, induced by $\left\{V_1,V_2, \cdots, V_{i}\right\}$ and $\left\{V_{i+1},V_{i+2}, \cdots V_m\right\}$ respectively. Since $u_{(^1_2)}v_i$ and $u_4v_{(^{i-1}_{i})} \in M$, vertex $u_5v_i \in V_{sb}(G^{|n(1)|})$, and therefore, $|V_{st}G^{|n(1)|}|=10t+4$, where $|n(1)|=4t+2$. Also, only $u_3v_{i+1}$ and $u_5v_{i+1}$ are saturable on $V_{i+1}$.
 Suppose without loss of generality, that both $u_3v_{i+1}$ and $u_5v_{i+1}$ are saturated and thus, $u_3v_{(^{i+1}_{i+2})}, u_5v_{(^{i+1}_{i+2})} \in M$. Now, suppose that $G^{|n(4)|}$ is induced by $\left\{V_{i+3},V_{i+4}, \cdots V_m\right\}$, with $|n(4)|=4(k-t-1)+3$. By following the techniques employed earlier, it can be shown that $|V_{st}(G)| \leq |V_{st}(G^{|n(1)|})|+|V_{st}(G^{|n(2)|})| \leq 10k+4.$ The $u_4v_{(^{i}_{i+4})}$ case, has the same proof as the $u_4v_{(^{i-1}_{i})}$ case.
 \end{proof}
\begin{rem}There can be only one edge $u_{(^1_2)}v_i \in M$ for which $M$ is $MIM$ of $G_{5,m}$, if $M$ contains $u_{(^1_2)}v_i$ and $u_4v_{(^{i-1}_{i})}$ (or $u_4v_{(^{i}_{i+1})}$), and in this case, $i \equiv 0 \mod 4$ as shown in Figure 2.
\end{rem}

\begin{rem} It should be noted that the proof of the $i\equiv 1 \mod 4$ in Lemma \ref{lema4} will hold for $i\equiv 3 \mod 4$ by flipping the grid from right to left.
\end{rem}

{\tiny{
\begin{center}
\pgfdeclarelayer{nodelayer}
\pgfdeclarelayer{edgelayer}
\pgfsetlayers{nodelayer,edgelayer}
\begin{tikzpicture}
	\begin{pgfonlayer}{nodelayer}
	
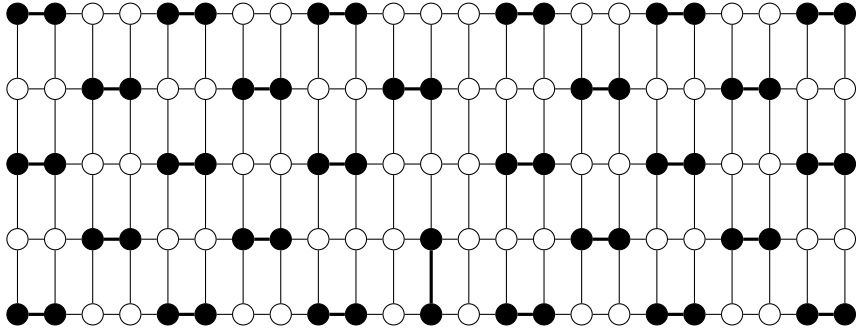
	\node [minimum size=0cm,]  at (-6.6,6.5) {Figure 2. A $G \equiv G_{5,23}$ Grid with $MIM_G=28$, $u_{^1_2}v_i, i \equiv 0 \mod 4$};
	
	  \node [minimum size=0cm,draw,fill=black,circle] (1) at (-12,7) {};
		\node [minimum size=0cm,draw,fill=black,circle] (2) at  (-11.5,7) {};
		\node [minimum size=0cm,draw,circle] (3) at (-11,7) {};
		\node [minimum size=0cm,draw,circle] (4) at (-10.5,7) {};
		\node [minimum size=0cm,draw,fill=black,circle] (5) at (-10,7) {};
		\node [minimum size=0cm,draw,fill=black,circle] (6) at (-9.5,7) {};
		\node [minimum size=0cm,draw,circle] (7) at (-9,7) {};
		\node [minimum size=0cm,draw,circle] (8) at (-8.5,7) {};
		\node [minimum size=0cm,draw,fill=black,circle] (9) at (-8,7) {};
		\node [minimum size=0cm,draw,fill=black,circle] (10) at (-7.5,7) {};
    \node [minimum size=0cm,draw,circle] (11) at (-7,7) {};
		\node [minimum size=0cm,draw,fill=black,circle] (12) at (-6.5,7) {};
		\node [minimum size=0cm,draw,circle] (13) at (-6,7) {};
		\node [minimum size=0cm,draw,fill=black,circle] (14) at (-5.5,7) {};
		\node [minimum size=0cm,draw,fill=black,circle] (15) at (-5,7) {};
		\node [minimum size=0cm,draw,circle] (16) at (-4.5,7) {};
		\node [minimum size=0cm,draw,circle] (17) at (-4,7) {};
		\node [minimum size=0cm,draw,fill=black,circle] (18) at (-3.5,7) {};
		\node [minimum size=0cm,draw,fill=black,circle] (19) at (-3,7) {};
		\node [minimum size=0cm,draw,circle] (20) at (-2.5,7) {};
		\node [minimum size=0cm,draw,circle] (21) at (-2,7) {};
		\node [minimum size=0cm,draw,fill=black,circle] (22) at (-1.5,7) {};
		\node [minimum size=0cm,draw,fill=black,circle] (23) at (-1,7) {};

		\node [minimum size=0cm,draw,circle] (24) at (-12,8) {};
		\node [minimum size=0cm,draw,circle] (25) at (-11.5,8) {};
		\node [minimum size=0cm,draw,fill=black,circle] (26) at (-11,8) {};
		\node [minimum size=0cm,draw,fill=black,circle] (27) at (-10.5,8) {};
		\node [minimum size=0cm,draw,circle] (28) at (-10,8) {};
		\node [minimum size=0cm,draw,circle] (29) at (-9.5,8) {};
		\node [minimum size=0cm,draw,fill=black,circle] (30) at (-9,8) {};
		\node [minimum size=0cm,draw,fill=black,circle] (31) at (-8.5,8) {};
		\node [minimum size=0cm,draw,circle] (32) at (-8,8) {};
		\node [minimum size=0cm,draw,circle] (33) at (-7.5,8) {};
		\node [minimum size=0cm,draw,circle] (34) at (-7,8) {};
		\node [minimum size=0cm,draw,fill=black,circle] (35) at (-6.5,8) {};
		\node [minimum size=0cm,draw,circle] (36) at (-6,8) {};
		\node [minimum size=0cm,draw,circle] (37) at (-5.5,8) {};
		\node [minimum size=0cm,draw,circle] (38) at (-5,8) {};
		\node [minimum size=0cm,draw,fill=black,circle] (39) at (-4.5,8) {};
		\node [minimum size=0cm,draw,fill=black,circle] (40) at (-4,8) {};
		\node [minimum size=0cm,draw,circle] (41) at (-3.5,8) {};
		\node [minimum size=0cm,draw,circle] (42) at (-3,8) {};
		\node [minimum size=0cm,draw,fill=black,circle] (43) at (-2.5,8) {};
		\node [minimum size=0cm,draw,fill=black,circle] (44) at (-2,8) {};
		\node [minimum size=0cm,draw,circle] (45) at (-1.5,8) {};
		\node [minimum size=0cm,draw,circle] (46) at (-1,8) {};

		\node [minimum size=0cm,draw,fill=black,circle] (47) at (-12,9) {};
		\node [minimum size=0cm,draw,fill=black,circle] (48) at (-11.5,9) {};
		\node [minimum size=0cm,draw,circle] (49) at (-11,9) {};
		\node [minimum size=0cm,draw,circle] (50) at (-10.5,9) {};
		\node [minimum size=0cm,draw,fill=black,circle] (51) at (-10,9) {};
		\node [minimum size=0cm,draw,fill=black,circle] (52) at (-9.5,9) {};
		\node [minimum size=0cm,draw,circle] (53) at (-9,9) {};
		\node [minimum size=0cm,draw,circle] (54) at (-8.5,9) {};
		\node [minimum size=0cm,draw,fill=black,circle] (55) at (-8,9) {};
		\node [minimum size=0cm,draw,fill=black,circle] (56) at (-7.5,9) {};

		\node [minimum size=0cm,draw,circle] (57) at (-7,9) {};
		\node [minimum size=0cm,draw,circle] (58) at (-6.5,9) {};
		\node [minimum size=0cm,draw,circle] (59) at (-6,9) {};
		\node [minimum size=0cm,draw,fill=black,circle] (60) at (-5.5,9) {};
		\node [minimum size=0cm,draw,fill=black,circle] (61) at (-5,9) {};
		\node [minimum size=0cm,draw,circle] (62) at (-4.5,9) {};
		\node [minimum size=0cm,draw,circle] (63) at (-4,9) {};
		\node [minimum size=0cm,draw,fill=black,circle] (64) at (-3.5,9) {};
		\node [minimum size=0cm,draw,fill=black,circle] (65) at (-3,9) {};
		\node [minimum size=0cm,draw,circle] (66) at (-2.5,9) {};
		\node [minimum size=0cm,draw,circle] (67) at (-2,9) {};
		\node [minimum size=0cm,draw,fill=black,circle] (68) at (-1.5,9) {};
		\node [minimum size=0cm,draw,fill=black,circle] (69) at (-1,9) {};

		\node [minimum size=0cm,draw,circle] (70) at (-12,10) {};
		\node [minimum size=0cm,draw,circle] (71) at (-11.5,10) {};
		\node [minimum size=0cm,draw,fill=black,circle] (72) at (-11,10) {};
		\node [minimum size=0cm,draw,fill=black,circle] (73) at (-10.5,10) {};
		\node [minimum size=0cm,draw,circle] (74) at (-10,10) {};
		\node [minimum size=0cm,draw,circle] (75) at (-9.5,10) {};
		\node [minimum size=0cm,draw,fill=black,circle] (76) at (-9,10) {};
		\node [minimum size=0cm,draw,fill=black,circle] (77) at (-8.5,10) {};
		\node [minimum size=0cm,draw,circle] (78) at (-8,10) {};
		\node [minimum size=0cm,draw,circle] (79) at (-7.5,10) {};

		\node [minimum size=0cm,draw,fill=black,circle] (80) at (-7,10) {};
		\node [minimum size=0cm,draw,fill=black,circle] (81) at (-6.5,10) {};
		\node [minimum size=0cm,draw,circle] (82) at (-6,10) {};
		\node [minimum size=0cm,draw,circle] (83) at (-5.5,10) {};
		\node [minimum size=0cm,draw,circle] (84) at (-5,10) {};
		\node [minimum size=0cm,draw,fill=black,circle] (85) at (-4.5,10) {};
		\node [minimum size=0cm,draw,fill=black,circle] (86) at (-4,10) {};
		\node [minimum size=0cm,draw,circle] (87) at (-3.5,10) {};
		\node [minimum size=0cm,draw,circle] (88) at (-3,10) {};
		\node [minimum size=0cm,draw,fill=black,circle] (89) at (-2.5,10) {};
		\node [minimum size=0cm,draw,fill=black,circle] (90) at (-2,10) {};
		\node [minimum size=0cm,draw,,circle] (91) at (-1.5,10) {};
		\node [minimum size=0cm,draw,circle] (92) at (-1,10) {};

		\node [minimum size=0cm,draw,fill=black,circle] (93) at (-12,11) {};
		\node [minimum size=0cm,draw,fill=black,circle] (94) at (-11.5,11) {};
		\node [minimum size=0cm,draw,circle] (95) at (-11,11) {};
		\node [minimum size=0cm,draw,circle] (96) at (-10.5,11) {};
		\node [minimum size=0cm,draw,fill=black,circle] (97) at (-10,11) {};
		\node [minimum size=0cm,draw,fill=black,circle] (98) at (-9.5,11) {};
		\node [minimum size=0cm,draw,circle] (99) at (-9,11) {};
		\node [minimum size=0cm,draw,circle] (100) at (-8.5,11) {};
		\node [minimum size=0cm,draw,fill=black,circle] (101) at (-8,11) {};
		\node [minimum size=0cm,draw,fill=black,circle] (102) at (-7.5,11) {};
		\node [minimum size=0cm,draw,circle] (103) at (-7,11) {};
		\node [minimum size=0cm,draw,circle] (104) at (-6.5,11) {};
		\node [minimum size=0cm,draw,circle] (105) at (-6,11) {};
		\node [minimum size=0cm,draw,fill=black,circle] (106) at (-5.5,11) {};
		\node [minimum size=0cm,draw,fill=black,circle] (107) at (-5,11) {};
		\node [minimum size=0cm,draw,circle] (108) at (-4.5,11) {};
		\node [minimum size=0cm,draw,circle] (109) at (-4,11) {};
		\node [minimum size=0cm,draw,fill=black,circle] (110) at (-3.5,11) {};
		\node [minimum size=0cm,draw,fill=black,circle] (111) at (-3,11) {};
		\node [minimum size=0cm,draw,circle] (112) at (-2.5,11) {};
		\node [minimum size=0cm,draw,circle] (113) at (-2,11) {};
		\node [minimum size=0cm,draw,fill=black,circle] (114) at (-1.5,11) {};
		\node [minimum size=0cm,draw,fill=black,circle] (115) at (-1,11) {};

			\end{pgfonlayer}
				\begin{pgfonlayer}{edgelayer}
		\draw [very thick=1.00] (1) to (2);
		\draw [thin=1.00] (2) to (3);
		\draw [thin=1.00] (3) to (4);
		\draw [thin=1.00] (4) to (5);
		\draw [very thick=1.00] (5) to (6);
		\draw [thin=1.00] (6) to (7);
		\draw [thin=1.00] (7) to (8);
		\draw [thin=1.00] (8) to (9);
		\draw [very thick=1.00] (9) to (10);
		\draw [thin=1.00] (10) to (11);
		\draw [thin=1.00] (11) to (12);
		\draw [thin=1.00] (12) to (13);
		\draw [thin=1.00] (13) to (14);
		\draw [very thick=1.00] (14) to (15);
		\draw [thin=1.00] (15) to (16);
		\draw [thin=1.00] (16) to (17);
		\draw [thin=1.00] (17) to (18);
		\draw [very thick=1.00] (18) to (19);
		\draw [thin=1.00] (19) to (20);
		\draw [thin=1.00] (20) to (21);
		\draw [thin=1.00] (21) to (22);
		\draw [very thick=1.00] (22) to (23);

		\draw [thin=1.00] (24) to (25);
		\draw [thin=1.00] (25) to (26);
		\draw [very thick=1.00] (26) to (27);
		\draw [thin=1.00] (27) to (28);
		\draw [thin=1.00] (28) to (29);
		\draw [thin=1.00] (29) to (30);
		\draw [very thick=1.00] (30) to (31);
		\draw [thin=1.00] (31) to (32);
		\draw [thin=1.00] (32) to (33);
		\draw [thin=1.00] (33) to (34);
		\draw [thin=1.00] (34) to (35);
		\draw [thin=1.00] (35) to (36);
		\draw [thin=1.00] (36) to (37);
		\draw [thin=1.00] (37) to (38);
		\draw [thin=1.00] (38) to (39);
		\draw [very thick=1.00] (39) to (40);
		\draw [thin=1.00] (40) to (41);
		\draw [thin=1.00] (41) to (42);
		\draw [thin=1.00] (42) to (43);
		\draw [very thick=1.00] (43) to (44);
		\draw [thin=1.00] (44) to (45);
		\draw [thin=1.00] (45) to (46);
		
		\draw [very thick=1.00] (47) to (48);
		\draw [thin=1.00] (48) to (49);
		\draw [thin=1.00] (49) to (50);
		\draw [thin=1.00] (50) to (51);
		\draw [very thick=1.00] (51) to (52);
		\draw [thin=1.00] (52) to (53);
		\draw [thin=1.00] (53) to (54);
		\draw [thin=1.00] (54) to (55);
		\draw [very thick=1.00] (55) to (56);
		\draw [thin=1.00] (56) to (57);
		\draw [thin=1.00] (57) to (58);
		\draw [thin=1.00] (58) to (59);
		\draw [thin=1.00] (59) to (60);
		\draw [very thick=1.00] (60) to (61);
		\draw [thin=1.00] (61) to (62);
		\draw [thin=1.00] (62) to (63);
		\draw [thin=1.00] (63) to (64);
		\draw [very thick=1.00] (64) to (65);
		\draw [thin=1.00] (65) to (66);
		\draw [thin=1.00] (66) to (67);
		\draw [thin=1.00] (67) to (68);
		\draw [very thick=1.00] (68) to (69);

   	\draw [thin=1.00] (70) to (71);
		\draw [thin=1.00] (71) to (72);
		\draw [very thick=1.00] (72) to (73);
		\draw [thin=1.00] (73) to (74);
		\draw [thin=1.00] (74) to (75);
		\draw [thin=1.00] (75) to (76);
		\draw [very thick=1.00] (76) to (77);
		\draw [thin=1.00] (77) to (78);
		\draw [thin=1.00] (78) to (79);
		\draw [thin=1.00] (79) to (80);
		\draw [very thick=1.00] (80) to (81);
		\draw [thin=1.00] (81) to (82);
		\draw [thin=1.00] (82) to (83);
		\draw [thin=1.00] (83) to (84);
		\draw [thin=1.00] (84) to (85);
		\draw [very thick=1.00] (85) to (86);
		\draw [thin=1.00] (86) to (87);
		\draw [thin=1.00] (87) to (88);
		\draw [thin=1.00] (88) to (89);
		\draw [very thick=1.00] (89) to (90);
		\draw [thin=1.00] (90) to (91);
		\draw [thin=1.00] (91) to (92);

		\draw [very thick=1.00] (93) to (94);
		\draw [thin=1.00] (94) to (95);
		\draw [thin=1.00] (95) to (96);
		\draw [thin=1.00] (96) to (97);
		\draw [very thick=1.00] (97) to (98);
		\draw [thin=1.00] (98) to (99);
		\draw [thin=1.00] (99) to (100);
		\draw [thin=1.00] (100) to (101);
		\draw [very thick=1.00] (101) to (102);
		\draw [thin=1.00] (102) to (103);
		\draw [thin=1.00] (103) to (104);
		\draw [thin=1.00] (104) to (105);
		\draw [thin=1.00] (105) to (106);
		\draw [very thick=1.00] (106) to (107);
		\draw [thin=1.00] (107) to (108);
		\draw [thin=1.00] (108) to (109);
		\draw [thin=1.00] (109) to (110);
		\draw [very thick=1.00] (110) to (111);
		\draw [thin=1.00] (111) to (112);
		\draw [thin=1.00] (112) to (113);
		\draw [thin=1.00] (113) to (114);
		\draw [very thick=1.00] (114) to (115);

		\draw [thin=1.00] (1) to (24);
		\draw [thin=1.00] (24) to (47);
		\draw [thin=1.00] (47) to (70);
		\draw [thin=1.00] (70) to (93);
		
		\draw [thin=1.00] (2) to (25);
		\draw [thin=1.00] (25) to (48);
		\draw [thin=1.00] (48) to (71);
		\draw [thin=1.00] (71) to (94);
		
		\draw [thin=1.00] (3) to (26);
		\draw [thin=1.00] (26) to (49);
		\draw [thin=1.00] (49) to (72);
		\draw [thin=1.00] (72) to (95);
		
		\draw [thin=1.00] (4) to (27);
		\draw [thin=1.00] (27) to (50);
		\draw [thin=1.00] (50) to (73);
		\draw [thin=1.00] (73) to (96);
		
		\draw [thin=1.00] (5) to (28);
		\draw [thin=1.00] (28) to (51);
		\draw [thin=1.00] (51) to (74);
		\draw [thin=1.00] (74) to (97);
		
		\draw [thin=1.00] (6) to (29);
		\draw [thin=1.00] (29) to (52);
		\draw [thin=1.00] (52) to (75);
		\draw [thin=1.00] (75) to (98);
		
		\draw [thin=1.00] (7) to (30);
		\draw [thin=1.00] (30) to (53);
		\draw [thin=1.00] (53) to (76);
		\draw [thin=1.00] (76) to (99);
		
		\draw [thin=1.00] (8) to (31);
		\draw [thin=1.00] (31) to (54);
		\draw [thin=1.00] (54) to (77);
		\draw [thin=1.00] (77) to (100);
		
		\draw [thin=1.00] (9) to (32);
		\draw [thin=1.00] (32) to (55);
		\draw [thin=1.00] (55) to (78);
		\draw [thin=1.00] (78) to (101);
		
		\draw [thin=1.00] (10) to (33);
		\draw [thin=1.00] (33) to (56);
		\draw [thin=1.00] (56) to (79);
		\draw [thin=1.00] (79) to (102);
		
		\draw [thin=1.00] (11) to (34);
		\draw [thin=1.00] (34) to (57);
		\draw [thin=1.00] (57) to (80);
		\draw [thin=1.00] (80) to (103);

		\draw [very thick=1.00] (12) to (35);
		\draw [thin=1.00] (35) to (58);
		\draw [thin=1.00] (58) to (81);
		\draw [thin=1.00] (81) to (104);

		\draw [thin=1.00] (13) to (36);
		\draw [thin=1.00] (36) to (59);
		\draw [thin=1.00] (59) to (82);
		\draw [thin=1.00] (82) to (105);
		
		\draw [thin=1.00] (14) to (37);
		\draw [thin=1.00] (37) to (60);
		\draw [thin=1.00] (60) to (83);
		\draw [thin=1.00] (83) to (106);
		
		\draw [thin=1.00] (15) to (38);
		\draw [thin=1.00] (38) to (61);
		\draw [thin=1.00] (61) to (84);
		\draw [thin=1.00] (84) to (107);
		
		\draw [thin=1.00] (16) to (39);
		\draw [thin=1.00] (39) to (62);
		\draw [thin=1.00] (62) to (85);
		\draw [thin=1.00] (85) to (108);
		
	  \draw [thin=1.00] (17) to (40);
		\draw [thin=1.00] (40) to (63);
		\draw [thin=1.00] (63) to (86);
		\draw [thin=1.00] (86) to (109);
		
		\draw [thin=1.00] (18) to (41);
		\draw [thin=1.00] (41) to (64);
		\draw [thin=1.00] (64) to (87);
		\draw [thin=1.00] (87) to (110);
		
		\draw [thin=1.00] (19) to (42);
		\draw [thin=1.00] (42) to (65);
		\draw [thin=1.00] (65) to (88);
		\draw [thin=1.00] (88) to (111);
		
		\draw [thin=1.00] (20) to (43);
		\draw [thin=1.00] (43) to (66);
		\draw [thin=1.00] (66) to (89);
		\draw [thin=1.00] (89) to (112);
		
		\draw [thin=1.00] (21) to (44);
		\draw [thin=1.00] (44) to (67);
		\draw [thin=1.00] (67) to (90);
		\draw [thin=1.00] (90) to (113);
		
		\draw [thin=1.00] (22) to (45);
		\draw [thin=1.00] (45) to (68);
		\draw [thin=1.00] (68) to (91);
		\draw [thin=1.00] (91) to (114);
		
		\draw [thin=1.00] (23) to (46);
		\draw [thin=1.00] (46) to (69);
		\draw [thin=1.00] (69) to (92);
		\draw [thin=1.00] (92) to (115);

	\end{pgfonlayer}
\end{tikzpicture}
\end{center}
}}

The previous results and remarks yield the following conclusion.

\begin{cor}\label{coro1} Suppose that $m \geq 23$ and $M$ is the $MIM$ of $G$, some $G_{5,m}$ grid. Then, if for at most some positive integer $i$, $1<i<m$, $u_{({^1_2})}v_i \in M$, then, $i\equiv 0\mod 4$.
\end{cor}

\begin{lem}\label{lema6} Let $M$ be a matching of $G_{5,m}$ with $m \equiv 3 \mod 4$ and let $u_{(^1_2)}v_i, u_{(^1_2)}v_j \in M$, $1 < i< j <m$, such that $i \equiv 0 \mod 4$ and $j \equiv 0 \mod 4$, then $M$ is not an $MIM$ of $G$.
\end{lem}

 The claim in Lemma \ref{lema6} can easily be proved using earlier techniques and Lemma \ref{lema1} and Remark \ref{rema1}.

\begin{rem} It should be noted from the previous results and from Corollary \ref{coro1} that if $M$ is the $MIM$ of $G_{5,m}$, $m \equiv 3 \mod 4$, then at most, $M$ contains two edges of the form $u_{(^1_2)}v_i$, $u_{(^1_2)}v_j$ and $j$ can only be $4$ when $i=1$ or $i$ can only be $m-3$ when $j=m$.
\end{rem}
\begin{thm}\label{thm1} Let $M$ be the $MIM$ of $G$, a $G_{5,m}$ grid and let $M$ contain $u_{(^1_2)}v_1$ and $u_{(^1_2)}v_4$ (or $u_{(^1_2)}v_{m-3}$ and $u_{(^1_2)}v_m$). Then there are at least $2k+2$ saturated vertices on $U_1 \subset G$.
\end{thm}
\begin{proof} For $u_{(^1_2)}v_1$ and $u_{(^1_2)}v_4$ to be in $M$, either $u_{(^4_5)}v_4 \in M$ or $u_5v_{(^3_4)} \in M$. Now, let $\left\{V_6,V_7, \cdots,V_m\right\}$ induce $G^{|m(1)|} \subset G$. Clearly, $|m(1)|\equiv 2 \mod 4$ and $|V_{st}(G^{|m(1)|})|=10k-4$. Let $G^{|m(1)|}\backslash \left\{u_1v_6,u_1v_7,\cdots,u_1v_m\right\}$ induce $G^{|m(2)|} \subset G^{|m(1)|}$. Then, $G^{|m(2)|}$ is a $G_{4,m-5}$ subgrid of $G^{|m(1)|}$. Now, $|V_{st}(G^{|m(2)|})| \leq 8k-4$. Thus for $V(U_1) \subset V(G^{|m(1)|})$, $|V(U)| \geq 2k$. Thus, $U_1$ contains at least $2k+2$ (i.e, $\frac{m-1}{2}$) saturated vertices.
\end{proof}

Next we investigate $G_{3,m}$, where $m \equiv 3 \mod 4$.

\begin{lem} Suppose that $G$ is a $G_{3,m}$ grid with $m \equiv 3 \mod 4$ and $M$ is an induced matching of $G_{3,m}$, with $\left\{u_{(^1_2)}v_i, u_{(^1_2)}v_{i+2},u_{(^1_2)}v_j, u_{(^1_2)}v_{j+2}\right\} \in M$ and $i+2 \geq j$. Then $M$ is not a $MIM$ of $G$,
\end{lem}
 \begin{proof} Suppose $i+2 \geq j$. Since $m=4k+3$, $|V_{sb}(G)|=6k+5$ and $|V_{st}(G)|=6k+4$. Thus, $G$ contain at most one $FSV$. Now from the conditions in the hypothesis, it is clear that $u_3v_{i+1}$ and $u_3v_{j+1}$ are $FSV$s in $G$, which is a contradiction. Same argument hold if $i+2=j$ since both $u_3v_{i+1}$ and $u_3v_{i+3}$ are $FSV$s in $G$.
 \end{proof}

\begin{rem}\label{rem1} Suppose that $G_n$ is $G_{3,n}$, a subgrid of $G_{3,m}$ and induced by $\left\{V_{i+1}, V_{i+2}, \cdots, V_{i+n}\right\}$ and $G'$ is a subgraph of $G$, with $G'=G_n+\left\{u_3v_i,u_3v_{i+n+1}\right\}$, then the following are easy to verify. For
\begin{enumerate}
\item $n \equiv 0 \mod 4$, $|V_{st}(G')| \leq |V_{sb}(G_n)|+2$
\item $n \equiv 1 \mod 4$, $|V_{st}(G')| \leq |V_{sb}(G_n)|+2$
\item $n \equiv 2 \mod 4$, $|V_{st}(G')|=|V_{sb}(G_n)|$
\item $n \equiv 3 \mod 4$, $|V_{st}(G')| \leq |V_{sb}(G_n)|+1$
\end{enumerate}
\end{rem}

\begin{lem}\label{lema9} Let $u_{(^1_2)}v_j, u_{(^1_2)}v_{j+3}, u_{(^1_2)}v_k, u_{(^1_2)}v_{k+3}, u_{(^1_2)}v_l, u_{(^1_2)}v_{l+3}$ be in $M$an induced matching of $G$ a $G_{3,m}$ grid and $m \equiv 3 \mod 4$. Then $M$ is not $MIM$ of $G$.
\end{lem}
\begin{proof} Case 1: Let $m=4p+3$, $j+3=k$ and $l=k+3$. Suppose $G^{|m(1)|}$ is a subgraph of $G$, induced by $\left\{V_{j-1}, V_j, \cdots, V_{i+4}\right\}$. Then $|m(1)|=12$, with $u_3v_{j-1}$ and $u_3v_{i+4}$ as $FSV$s. For one of $u_3v_{j-1}$ and $u_3v_{i+4}$ to be relevant for $M$ to be $MIM$ of $G$, say $u_3v_{j-1}$, then for $G^{|m(2)|}$, induced by $\left\{V_1, V_2, \cdots, V_{j-2}\right\}$, $|V_{sb}(G^{|m(2)|})|$ must be odd, which can only be if $j-2 \equiv 3 \mod 4$. So, suppose $j-2 \equiv 3 \mod 4$, then $|V_{st}(G^{|m(2)|})+u_3v_{j-1}| \leq |V_{sb}(G^{|m(2)|})|+1=6q+6$, where $|m(2)|=4q+3$, for $q \geq 1$, since $|m(1)|=12$ and $|n(2)| \equiv 3 \mod 4$. Now let $G^{|m(3)|}=G^{|m(1)|} \cup G^{|m(2)|}$, where $|m(3)|=|m(1)|+|m(2)| \equiv 3 \mod 4$ and $G^{|m(4)|} \subset G$ be defined as a subgrid of $G$ induced by $\left\{V_{i+5}, V_{i+6}, \cdots, V_m\right\}$. Clearly, $|m(4)|\equiv 0 \mod 4$. Since $|V_{sb}(G^{|m(4)|})|= |V_{st}(G^{|m(4)|})|$, which is even, then $|V_{st}(G^{|m(4)|}+u_3v_{i+4})|=|V_{st}(G^{|m(4)|})|=6p-6q-18$. Now, it can be seen that $|V_{st}(G^{|m(1)|}) \backslash \left\{u_3v_{j-1}, u_3v_{l+4}\right\}|=14$. Therefore, $|V_{st}(G)| \leq 6p+2$ instead of $6p+4$, and hence a contradiction.\\
Case 2: Suppose that $j+3 < k$ and $k+3 < l$. As in Case 1 and without loss of generality, let $j-2 \equiv 3 \mod 4$ and let $G^{|m(2)|}$ still be induced by $\left\{V_1, V_2, \cdots, V_{j-2}\right\}$. Also, let $G^{|m(4)|}$ be induced by $\left\{V_{l+5}, V_{l+6}, \cdots, V_m\right\}$, and set $|m(4)| \equiv 3 \mod 4$. Thus, $u_3v_{j-1}$ and $u_3v_{i+4}$ are both relevant for $M$ to be a $MIM$ of $G$ and $|V_{st}(G^{|m(2)|}+V_{j-1})|=|V_{sb}(G^{|m(2)|})|+1$ and $|V_{st}(G^{|m(4)|}+V_{l+4})|=|V_{sb}(G^{|m(4)|})|+1$. Set $G^{|m(2)|}+V_{j-1}=G^{|m(2^+)|}$ and set $G^{|m(4)|}+V_{i+4}=G^{|m(4^+)|}$ and let $\left\{V_{j},V_{j+1}, V_{j+2},V_{j+3}\right\}$ induce $G^{|m(5)|}$ while $\left\{V_i,V_{i+1},V_{i+2}, V_{i+3}\right\}$ induces $G^{|m(6)|}$. Furthermore, let $G^{|m(5^+)|}=G^{|m(5)|}+V_{j+4}$ and $G^{|m(6^+)|}$ contain, say, $h$ columns of $V_i$ in all, where $h \equiv 2 \mod 4$. Therefore, for $G^{|(m(7))|}= G \backslash \left\{G^{|m(2^+)|} \cup G^{|m(4^+)|} \cup G^{|m(5^+)|} \cup G^{|m(6^+)|} \right\}$, $|m(7)|=m-h = b \equiv 1 \mod 4$. Let $b=4a+1$, for some positive integer $a$ and let $G^{|m(4)|} \subset G^{|m(7)|}$, where $G^{|m(7)|}$ is induced by $\left\{V_k, V_{k+1}, V_{k_2}, V_{k+3}\right\}$. Certainly, $u_3v_{k-1}, u_3v_{k+4},u_3v_{j+4}, u_3v_{l-1} \in V_{sb}(G)$. Now, let $G^{|(4)|}$ be induced by $\left\{V_k, V_{k+1}, V_{k+2}, V_{k+3}\right\}$ and $G^{|4^{++}|}$ be induced by $G^{|(4)|}+ \left\{V_{k-1},V_{k+4}\right\}$, with $|4++|=6$. So, $b-6 \equiv 3 \mod 4$, which is odd and thus can only be the sum of an even and an odd positive integer. Therefore, let $G^{|m(8)|}$ and $G^{|m(9)|}$ be induced by $\left\{V_{j+5}, V_{j+6}, \cdots, V_{k-2}\right\}$ and  $\left\{V_{j+5}, V_{j+6}, \cdots, V_{l-2}\right\}$, with $|m(8)|+|m(9)|=b$. Suppose thus, that $|m(8)| \equiv 0 \mod 4$, then, $|m(9)| \equiv 3 \mod 4$ and suppose $|m(8)| \equiv 1 \mod 4$, then $|m(9)| \equiv 2 \mod 4$. For $|m(8)| \equiv 0 \mod 4$, let $G^{|m(10)|}=G^{|m(2^+)|+|m(5^+)|}$ be $G^{|m(2^+)|} \cup G^{|m(5^+)|}$ and $G^{|m(11)|}=G^{|m(6^+)|+|m(4^+)|}$ be $G^{|m(6^+)|} \cup G^{|m(4^+)|}$, where $|m(2^+)|+|m(5^+)|=4q+9$ and $|m(4^+)|+|m(6^+)|=4r+9$, where $|m(4)|=4r+3$. Therefore, as defined, $b=|m(7)|=4p-4q-4r-15$ and thus $b-6=4(p-q-r-6)+3$.  Set $p-q-r-6=f$. Now, for $|m(8)|$ and $|m(9)|$, if $|m(8)|=4g$, for some positive integer $g$, then $|m(9)|=4(f-g)+3$. Next we sum the maximal values of the subgrid of $G$ as follows:
$|V_{st}(G)| \leq |V_{st}(G^{|m(2^+)|} \cup G^{|m(5)|})|+ |V_{st}(G^{|m(8)|}+\left\{u_3v_{j+4},u_3v_{k-1}\right\})|$ + $|V_{st}(G^{|m(4)|})+$$|V_{st}(G^{|m(9)|}+\left\{u_3v_{k+4},u_3v_{l-1}\right\})|+$ $|V_{st}(G^{|m(6)|}| \cup G^{|m(4^+)|})$$\leq 6p+2$, which is less than $6p+4$ and hence a contradiction. For $|m(8)| \equiv 1 \mod 4$, and $|m(9)| \equiv 2 \mod 4$, we have $|m(8)|=4g+1$ and hence $|m(9)|=4(f-g)+2$ and $|V_{st}(G^{|m(9)|}+ \left\{u_3v_{k+4},u_3v_{l-1}\right\})|=6(f-g)+4$ and thus, $|V_{st}(G)| \leq 6p+2$.\\
Case 3: Suppose $j+3=k$ or $k+3=i$. Without loss of generality, let $j+3=k.$ Suppose as in Case 2, $j-2 \equiv 3 \mod 4$ and $m-(i+4) \equiv 3 \mod 4$. Let $G^{|n(1)|} \subset G$, a $G_{3,9}$ subgrid of $G$ be induced by $\left\{V_{j-1},v_j, \cdots, V_{j+7}\right\}$. Then for $G^{|n(2)|}=G^{|m(2)|} \cup G^{|n(1)|}$, $|n(2)|=|m(2)|+|n(1)|$,  $|n(2)| \equiv 0 \mod 4$. Likewise, suppose $\left\{V_{i-1}, V_i, \cdots, V_m\right\}$
induces $G^{|n(3)|}$, for which $|n(3)| \equiv 1 \mod 4$. If $|n(2)|$ and $|n(3)|$ are $4q$ and $4r+1$ respectively, then $|n(4)|\equiv 2 \mod 4$. So far, $G^{|n(4)|}$, is induced by $\left\{V_{i+8},V_{i+9}, \cdots, V_{l-2}\right\}$ and by Remark \ref{rem1}, $|V_{st}(G^{|n(4)|})+\left\{u_3v_{j+7},u_3v_{l-1}\right\}|=|V_{sb}(G^{|n(4)|})|$. By a summation similar to the one at the end of case 2, $|V_{st}(G)| \leq |V_{st}G^{|n(2)|}|+ |V_{st}(G^{|n(4)|})|+ |V_{st}(G^{|n(3)|})|\leq 6p+2$.
\end{proof}

\begin{rem} By following the technique employed in Lemma \ref{lema9}, it can be established that given $u_{(^1_2)}v_i,u_{(^1_2)}v_{i+2} \in M$ and $u_{(^1_2)}v_j,u_{(^1_2)}v_{j+2} \in M$ of $G$, a $G_{3,m}$ grid, $m \equiv 3 \mod 4$, $i+2 \leq j$, then $M$ is not a $MIM$ of $G$.
\end{rem}
\begin{rem} Let $M$ be an induced matching of $G$, a $G_{3,m}$ grid, and $i$ be some fixed positive integer. Suppose $u_{(^1_2)}v_{i+8(n)} \in M$, for all non-negative integer $n$ for which $1 \leq i+8(n) \leq m$. Let $M$ be the maximum induced matching of $G$. Then,
\begin{enumerate}
	\item if $i > 1$, then $i-1$ is either $2,3,4$ or $6$
	\item if $i+8(n) < m$, for the maximum value of $n$, then $m-(i+8(n))$ is either $2,3,4$ or $6$.
\end{enumerate}
\end{rem}
Based on the results so far, we note that if $M$ is the $MIM$ of $G$, a $G_{3,m}$ grid, $m \equiv 3 \mod 4, m \geq 11$, the maximum number of edges of the type $u_{(^1_2)}v_k$ that is contained in $M$, $k$, a positive integer, is $k+2$ when $m =8k+3$ and $k+3$ when $m=8k+7$.

It can be easily established that for $H$ that is a $G_{k,m}$ grid, with $k \equiv 0 \mod 4$ and $m \equiv 3 \mod 4$, which is induced by $U_1,U_2,\cdots,U_k$, if $M_1$ is the $MIM$ of $H$, then, the least saturated vertices in $U_k$ is $\frac{m-1}{2}$. The next result describes the positions of the members of $M_1$ in $E(H)$ if $U_k$ contains $\frac{m-1}{2}$ saturated vertices.

\begin{lem}\label{lema10} Let $H$ be a $G_{k,m}$ grid with $k \equiv 0 \mod 4$ and $m \equiv 3 \mod 4$ and let $U_k$ contain the least possible, $\frac{m-1}{2}$, saturated vertices for which $N$ remains $MIM$ of $H$. Then, for any adjacent vertices $v',v''\in U_k$, edge $v'v'' \notin M$.
\end{lem}

\begin{proof}
 Induced by $\left\{U_1,U_2, \cdots, U_{k-2}\right\}$ and $\left\{U_{k-1}U_k\right\}$ respectively, let $G_1^{|m|}$ and $G_2^{|m|}$ be partitions of $H$ with $k-2 \equiv 2 \mod 4$. It can be seen that $|V_{st}(G_1^{|m|})|=|V_{sb}(G_1^{|m|})|=\frac{km-2m+2}{2}$. Since $|V_{st}{H}|=\frac{km}{2}$, then $|V_{st}(G_2^{|m|})| \leq m-1$. Now, let $G_3^{|m|}$ be a $G_{1,m}$ subgrid (a $P_m$ path) of $H$, induced by $U_{k}$. By the hypothesis, $U_k$ contains maximum of $\frac{m-1}{2}$ saturated vertices. Now, let $u_kv_i,u_kv_{i+1}$ be adjacent and saturated vertices of $G_3^{|m|}$. Then there are $\frac{m-5}{2}$ other saturated vertices on $G_3^{|m|}$. Without loss of generality, suppose that each of the remaining $\frac{m-5}{2}$ saturated vertices in $G_3^{|m|}$ is adjacent to some saturated vertex in $U_{k-1}$. Now, suppose $u_{k-1}v_j$ is a saturable vertex in $U_{k-1}$ and that $v \in V(H)$, such that $u_{k-1}v_jv \in M_1$. Now, $v \notin U_k$, since all the saturable vertices in  $U_k$ is saturated. Likewise, suppose $v \in U_{k-1}$ and then $u_{k-1}v_jv \in E(G_4^{|m|})$, where $G_4^{|m|}$ is a $G_{1,m}$ subgrid of $H$ induced by $U_{k-1}$. Then, clearly, at least one of $u_{k-1}v_j$ and $v$ is adjacent to a saturated vertex in $V_{st}(G_1^{|m|})$. Also, suppose that $v \in U_{k-2}$, since $|V_{sb}(G_1^{|m|})|=|V_{st}(G_1^{|m|})|$, then $|V_{st}(G_1^{|m|})|=|V_{st}(G_1^{|m|}+u_{k-1}u_j)|$. Hence, $v$ is a $FSV$ in $G_1^{m}$. Therefore, $|V_{st}{H}|\leq |V_{st}G_1^{|m|}|+|V_{st}G_2^{|m|}| \leq \frac{km-4}{2}$, which is a contradiction since $|V_{st}(H)|= \frac{km}{2}$, by \cite{RMG1}.

\end{proof}
\begin{rem}
The implication of Lemma \ref{lema10} is that for a grid $H' \subset H$, which is induced by $\left\{U_1,U_2, \cdots, U_{k-2}\right\} \subset V(H)$, $k-2 \equiv 2 \mod 4$, suppose $U_k$ contains the least possible saturated vertices, $\frac{m-1}{2}$, then $u_kv_2, u_kv_4,\cdots u_kv_{m-1}$ are saturated as shown in the example in Figure 3, for which $k=4$ and $m=7$.
\end{rem}
{\tiny{
\begin{center}
\pgfdeclarelayer{nodelayer}
\pgfdeclarelayer{edgelayer}
\pgfsetlayers{nodelayer,edgelayer}
\begin{tikzpicture}
	\begin{pgfonlayer}{nodelayer}
	
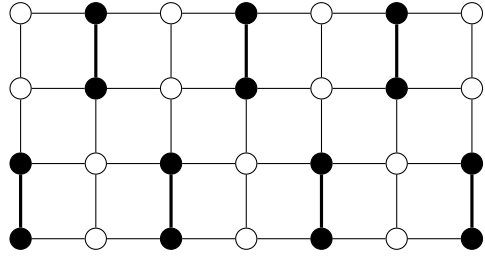
	\node [minimum size=0cm,]  at (-10,6.5) {Figure 3. A $G_{4,7}$ Grid with $MIM_G=11$};

		\node [minimum size=0cm,fill=black!,draw,circle] (1) at (-13,7) {};
		\node [minimum size=0cm,draw,circle] (2) at (-12,7) {};
		\node [minimum size=0cm,fill=black!,draw,circle] (3) at (-11,7) {};
		\node [minimum size=0cm,draw,circle] (4) at (-10,7) {};
		\node [minimum size=0cm,fill=black!,draw,circle] (5) at (-9,7) {};
		\node [minimum size=0cm,draw,circle] (6) at (-8,7) {};
		\node [minimum size=0cm,fill=black!,draw,circle] (7) at (-7,7) {};
		\node [minimum size=0cm,fill=black!,draw,circle] (11) at (-13,8) {};
		\node [minimum size=0cm,draw,circle] (12) at (-12,8) {};
		\node [minimum size=0cm,fill=black!,draw,circle] (13) at (-11,8) {};
		\node [minimum size=0cm,draw,circle] (14) at (-10,8) {};
		\node [minimum size=0cm,fill=black!,draw,circle] (15) at (-9,8) {};
		\node [minimum size=0cm,draw,circle] (16) at (-8,8) {};
		\node [minimum size=0cm,fill=black!,draw,circle] (17) at (-7,8) {};
	  \node [minimum size=0cm,draw,circle] (21) at (-13,9) {};
		\node [minimum size=0cm,fill=black!,draw,circle] (22) at (-12,9) {};
		\node [minimum size=0cm,draw,circle] (23) at (-11,9) {};
		\node [minimum size=0cm,fill=black!,draw,circle] (24) at (-10,9) {};
		\node [minimum size=0cm,draw,circle] (25) at (-9,9) {};
		\node [minimum size=0cm,fill=black!,draw,circle] (26) at (-8,9) {};
		\node [minimum size=0cm,draw,circle] (27) at (-7,9) {};
		\node [minimum size=0cm,draw,circle] (31) at (-13,10) {};
		\node [minimum size=0cm,fill=black!,draw,circle] (32) at (-12,10) {};
		\node [minimum size=0cm,draw,circle] (33) at (-11,10) {};
		\node [minimum size=0cm,fill=black!,draw,circle] (34) at (-10,10) {};
		\node [minimum size=0cm,draw,circle] (35) at (-9,10) {};
		\node [minimum size=0cm,fill=black!,draw,circle] (36) at (-8,10) {};
		\node [minimum size=0cm,draw,circle] (37) at (-7,10) {};

		\end{pgfonlayer}
	\begin{pgfonlayer}{edgelayer}
		\draw [thin=1.00] (1) to (2);
		\draw [thin=1.00] (2) to (3);
		\draw [thin=1.00] (3) to (4);
		\draw [thin=1.00] (4) to (5);
		\draw [thin=1.00] (5) to (6);
		\draw [thin=1.00] (6) to (7);
		\draw [thin=1.00] (11) to (12);
		\draw [thin=1.00] (12) to (13);
		\draw [thin=1.00] (13) to (14);
		\draw [thin=1.00] (14) to (15);
		\draw [thin=1.00] (15) to (16);
		\draw [thin=1.00] (16) to (17);
		\draw [thin=1.00] (21) to (22);
		\draw [thin=1.00] (22) to (23);
		\draw [thin=1.00] (23) to (24);
		\draw [thin=1.00] (24) to (25);
		\draw [thin=1.00] (25) to (26);
		\draw [thin=1.00] (26) to (27);
		
		\draw [thin=1.00] (31) to (32);
		\draw [thin=1.00] (32) to (33);
		\draw [thin=1.00] (33) to (34);
		\draw [thin=1.00] (34) to (35);
		\draw [thin=1.00] (35) to (36);
		\draw [thin=1.00] (36) to (37);

		\draw [very thick=1.00] (1) to (11);
		\draw [thin=1.00] (11) to (21);
		\draw [thin=1.00] (21) to (31);

		\draw [thin=1.00] (2) to (12);
		\draw [thin=1.00] (12) to(22);
		\draw [very thick=1.00] (22) to (32);

		\draw [very thick=1.00] (3) to (13);
		\draw [thin=1.00] (13) to (23);
		\draw [thin=1.00] (23) to (33);
		
		\draw [thin=1.00] (4) to (14);
		\draw [thin=1.00] (14) to (24);
		\draw [very thick=1.00] (24) to (34);
		
		\draw [very thick=1.00] (5) to (15);
		\draw [thin=1.00] (15) to (25);
		\draw [thin=1.00] (25) to (35);
		
		\draw [thin=1.00] (6) to  (16);
		\draw [thin=1.00] (16) to (26);
		\draw [very thick=1.00] (26) to (36);

		\draw [very thick=1.00] (7) to  (17);
		\draw [thin=1.00] (17) to (27);
		\draw [thin=1.00] (27) to (37);
		
		

	\end{pgfonlayer}
\end{tikzpicture}
\end{center}
}}

\begin{lem}\label{lema311} Let $G$ be a $G_{3,m}$ with an induced matching $M$ and $G^{|(9)|}$, induced by $\left\{V_i, V_{i+2}, \cdots, V_{i+8}\right\}$ be a $G_{3,9}$ subgrid of $G$. Suppose that $M'\subset M$ is an induced matching of $G^{|(9)|}$ such that $u_{(^1_2)}v_i,u_{(^1_2)}v_{i+8} \in M'$. No other edge $u_{^1_2}v_{i+t}, 1<t<i+7$ is contained in $M'$. Then for $G^{'|(9)|} \subset G^{|(9)|}$, defined as $G^{|(9)|}\backslash U_1$, $|V_{sb}(G^{'|(9)|})| \leq 8$.
\end{lem}

\begin{proof} Let $G^{|(7)|}=G^{|(9)|} \backslash \left\{\left\{u_1v_{i+1}, u_iv_{i+2}, \cdots, u_1v_{i+7}\right\},V_i,V_{i+8}\right\}$. It can be seen that $G^{|(7)|}$ is a $G_{2,7}$ subgrid of $G^{|(9)|}$. Clearly also, $G^{|(7)|} \subset G^{'|(9)|}$. Since $u_{^1_2}v_i,u_{^1_2v_{i+8}} \in M'$, then, $u_2v_{i+1}$ and $u_2v_{i+7}$ can not be saturated. Let $G_y \subset G^{|(7)|}$ be subgraph of $G^{|(7)|}$, defined as $G^{|(7)|} \backslash \left\{u_2v_{i+1},u_2v_{i+7}\right\}$. Now, $|V(G_y)|=12$ and $|V_{sb}(G_y)|$ can be seen to be at most $6$. Thus $|V_{sb}(G^{'|(9)|})|=|V_{sb}(G_y)|+2=8$, since $u_2v_i$ and $u_2v_{i+8}$ are saturated in $M'$.
\end{proof}

\begin{rem} For $U_1 \subset G^{|(9)|}$ as defined in Lemma \ref{lema311}, $U_1$ contains at least $6$ saturated vertices, implying that $M'$ contains two edges whose four vertices are from $U_1$.
\end{rem}

\begin{cor} Let $G$ be a $G_{3,m}$ grid with $m \geq 11$ and $m \equiv 3 \mod 4$. If $M$ is a $MIM$ of G. Then $M'$ contains at least $2k'$ edges from $U_1$, where $m=8k'+3$ or $m=8k'+7$.
\end{cor}

{\tiny{
\begin{center}
\pgfdeclarelayer{nodelayer}
\pgfdeclarelayer{edgelayer}
\pgfsetlayers{nodelayer,edgelayer}
\begin{tikzpicture}
	\begin{pgfonlayer}{nodelayer}
	
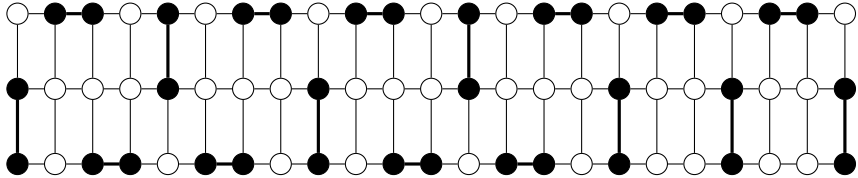
	\node [minimum size=0cm,]  at (-6.5,6.5) {Figure 4. A $G \equiv G_{3,23}$ Grid with $MIM_G=17$};
	
	  \node [minimum size=0cm,draw,fill=black,circle] (1) at (-12,7) {};
		\node [minimum size=0cm,draw,circle] (2) at (-11.5,7) {};
		\node [minimum size=0cm,draw,fill=black,circle] (3) at (-11,7) {};
		\node [minimum size=0cm,draw,fill=black,circle] (4) at (-10.5,7) {};
		\node [minimum size=0cm,draw,circle] (5) at (-10,7) {};
		\node [minimum size=0cm,draw,fill=black,circle] (6) at (-9.5,7) {};
		\node [minimum size=0cm,draw,fill=black,circle] (7) at (-9,7) {};
		\node [minimum size=0cm,draw,circle] (8) at (-8.5,7) {};
		\node [minimum size=0cm,draw,fill=black,circle] (9) at (-8,7) {};
		\node [minimum size=0cm,draw,circle] (10) at (-7.5,7) {};
    \node [minimum size=0cm,draw,fill=black,circle] (11) at (-7,7) {};
		\node [minimum size=0cm,draw,fill=black,circle] (12) at (-6.5,7) {};
		\node [minimum size=0cm,draw,circle] (13) at (-6,7) {};
		\node [minimum size=0cm,draw,fill=black,circle] (14) at (-5.5,7) {};
		\node [minimum size=0cm,draw,fill=black,circle] (15) at (-5,7) {};
		\node [minimum size=0cm,draw,circle] (16) at (-4.5,7) {};
		\node [minimum size=0cm,draw,fill=black,circle] (17) at (-4,7) {};
		\node [minimum size=0cm,draw,circle] (18) at (-3.5,7) {};
		\node [minimum size=0cm,draw,circle] (19) at (-3,7) {};
		\node [minimum size=0cm,draw,fill=black,circle] (20) at (-2.5,7) {};
		\node [minimum size=0cm,draw,circle] (21) at (-2,7) {};
		\node [minimum size=0cm,draw,circle] (22) at (-1.5,7) {};
		\node [minimum size=0cm,draw,fill=black,circle] (23) at (-1,7) {};

		\node [minimum size=0cm,draw,fill=black,circle] (24) at (-12,8) {};
		\node [minimum size=0cm,draw,circle] (25) at (-11.5,8) {};
		\node [minimum size=0cm,draw,circle] (26) at (-11,8) {};
		\node [minimum size=0cm,draw,circle] (27) at (-10.5,8) {};
		\node [minimum size=0cm,draw,fill=black,circle] (28) at (-10,8) {};
		\node [minimum size=0cm,draw,circle] (29) at (-9.5,8) {};
		\node [minimum size=0cm,draw,circle] (30) at (-9,8) {};
		\node [minimum size=0cm,draw,circle] (31) at (-8.5,8) {};
		\node [minimum size=0cm,draw,fill=black,circle] (32) at (-8,8) {};
		\node [minimum size=0cm,draw,circle] (33) at (-7.5,8) {};
		\node [minimum size=0cm,draw,circle] (34) at (-7,8) {};
		\node [minimum size=0cm,draw,circle] (35) at (-6.5,8) {};
		\node [minimum size=0cm,draw,fill=black,circle] (36) at (-6,8) {};
		\node [minimum size=0cm,draw,circle] (37) at (-5.5,8) {};
		\node [minimum size=0cm,draw,circle] (38) at (-5,8) {};
		\node [minimum size=0cm,draw,circle] (39) at (-4.5,8) {};
		\node [minimum size=0cm,draw,fill=black,circle] (40) at (-4,8) {};
		\node [minimum size=0cm,draw,circle] (41) at (-3.5,8) {};
		\node [minimum size=0cm,draw,circle] (42) at (-3,8) {};
		\node [minimum size=0cm,draw,fill=black,circle] (43) at (-2.5,8) {};
		\node [minimum size=0cm,draw,circle] (44) at (-2,8) {};
		\node [minimum size=0cm,draw,circle] (45) at (-1.5,8) {};
		\node [minimum size=0cm,draw,fill=black,circle] (46) at (-1,8) {};

		\node [minimum size=0cm,draw,circle] (47) at (-12,9) {};
		\node [minimum size=0cm,draw,fill=black,circle] (48) at (-11.5,9) {};
		\node [minimum size=0cm,draw,fill=black,circle] (49) at (-11,9) {};
		\node [minimum size=0cm,draw,circle] (50) at (-10.5,9) {};
		\node [minimum size=0cm,draw,fill=black,circle] (51) at (-10,9) {};
		\node [minimum size=0cm,draw,circle] (52) at (-9.5,9) {};
		\node [minimum size=0cm,draw,fill=black,circle] (53) at (-9,9) {};
		\node [minimum size=0cm,draw,fill=black,circle] (54) at (-8.5,9) {};
		\node [minimum size=0cm,draw,circle] (55) at (-8,9) {};
		\node [minimum size=0cm,draw,fill=black,circle] (56) at (-7.5,9) {};

		\node [minimum size=0cm,draw,fill=black,circle] (57) at (-7,9) {};
		\node [minimum size=0cm,draw,circle] (58) at (-6.5,9) {};
		\node [minimum size=0cm,draw,fill=black,circle] (59) at (-6,9) {};
		\node [minimum size=0cm,draw,circle] (60) at (-5.5,9) {};
		\node [minimum size=0cm,draw,fill=black,circle] (61) at (-5,9) {};
		\node [minimum size=0cm,draw,fill=black,circle] (62) at (-4.5,9) {};
		\node [minimum size=0cm,draw,circle] (63) at (-4,9) {};
		\node [minimum size=0cm,draw,fill=black,circle] (64) at (-3.5,9) {};
		\node [minimum size=0cm,draw,fill=black,circle] (65) at (-3,9) {};
		\node [minimum size=0cm,draw,circle] (66) at (-2.5,9) {};
		\node [minimum size=0cm,draw,fill=black,circle] (67) at (-2,9) {};
		\node [minimum size=0cm,draw,fill=black,circle] (68) at (-1.5,9) {};
		\node [minimum size=0cm,draw,circle] (69) at (-1,9) {};

		\end{pgfonlayer}
		\begin{pgfonlayer}{edgelayer}
		\draw [thin=1.00] (1) to (2);
		\draw [thin=1.00] (2) to (3);
		\draw [very thick=1.00] (3) to (4);
		\draw [thin=1.00] (4) to (5);
		\draw [thin=1.00] (5) to (6);
		\draw [very thick=1.00] (6) to (7);
		\draw [thin=1.00] (7) to (8);
		\draw [thin=1.00] (8) to (9);
		\draw [thin=1.00] (9) to (10);
		\draw [thin=1.00] (10) to (11);
		\draw [very thick=1.00] (11) to (12);
		\draw [thin=1.00] (12) to (13);
		\draw [thin=1.00] (13) to (14);
		\draw [very thick=1.00] (14) to (15);
		\draw [thin=1.00] (15) to (16);
		\draw [thin=1.00] (16) to (17);
		\draw [thin=1.00] (17) to (18);
		\draw [thin=1.00] (18) to (19);
		\draw [thin=1.00] (19) to (20);
		\draw [thin=1.00] (20) to (21);
		\draw [thin=1.00] (21) to (22);
		\draw [thin=1.00] (22) to (23);

			\draw [thin=1.00] (24) to (25);
		\draw [thin=1.00] (25) to (26);
		\draw [thin=1.00] (26) to (27);
		\draw [thin=1.00] (27) to (28);
		\draw [thin=1.00] (28) to (29);
		\draw [thin=1.00] (29) to (30);
		\draw [thin=1.00] (30) to (31);
		\draw [thin=1.00] (31) to (32);
		\draw [thin=1.00] (32) to (33);
			\draw [thin=1.00] (33) to (34);
		\draw [thin=1.00] (34) to (35);
		\draw [thin=1.00] (35) to (36);
		\draw [thin=1.00] (36) to (37);
		\draw [thin=1.00] (37) to (38);
		\draw [thin=1.00] (38) to (39);
		\draw [thin=1.00] (39) to (40);
		\draw [thin=1.00] (40) to (41);
		\draw [thin=1.00] (41) to (42);
		\draw [thin=1.00] (42) to (43);
		\draw [thin=1.00] (43) to (44);
		\draw [thin=1.00] (44) to (45);
		\draw [thin=1.00] (45) to (46);
		
			\draw [thin=1.00] (47) to (48);
		\draw [very thick=1.00] (48) to (49);
		\draw [thin=1.00] (49) to (50);
		\draw [thin=1.00] (50) to (51);
		\draw [thin=1.00] (51) to (52);
		\draw [thin=1.00] (52) to (53);
		\draw [very thick=1.00] (53) to (54);
		\draw [thin=1.00] (54) to (55);
		\draw [thin=1.00] (55) to (56);
		\draw [very thick=1.00] (56) to (57);
		\draw [thin=1.00] (57) to (58);
		\draw [thin=1.00] (58) to (59);
		\draw [thin=1.00] (59) to (60);
		\draw [thin=1.00] (60) to (61);
		\draw [very thick=1.00] (61) to (62);
		\draw [thin=1.00] (62) to (63);
		\draw [thin=1.00] (63) to (64);
		\draw [very thick=1.00] (64) to (65);
		\draw [thin=1.00] (65) to (66);
		\draw [thin=1.00] (66) to (67);
		\draw [very thick=1.00] (67) to (68);
		\draw [thin=1.00] (68) to (69);

		\draw [very thick=1.00] (1) to (24);
		\draw [thin=1.00] (24) to (47);

		\draw [thin=1.00] (2) to (25);
		\draw [thin=1.00] (25) to (48);

		\draw [thin=1.00] (3) to (26);
		\draw [thin=1.00] (26) to (49);

		\draw [thin=1.00] (4) to (27);
		\draw [thin=1.00] (27) to (50);

		\draw [thin=1.00] (5) to (28);
		\draw [very thick=1.00] (28) to (51);

		\draw [thin=1.00] (6) to (29);
		\draw [thin=1.00] (29) to (52);

		\draw [thin=1.00] (7) to (30);
		\draw [thin=1.00] (30) to (53);

		\draw [thin=1.00] (8) to (31);
		\draw [thin=1.00] (31) to (54);

			\draw [very thick=1.00] (9) to (32);
		\draw [thin=1.00] (32) to (55);

		\draw [thin=1.00] (10) to (33);
		\draw [thin=1.00] (33) to (56);

		\draw [thin=1.00] (11) to (34);
		\draw [thin=1.00] (34) to (57);

		\draw [thin=1.00] (12) to (35);
		\draw [thin=1.00] (35) to (58);

		\draw [thin=1.00] (13) to (36);
		\draw [very thick=1.00] (36) to (59);

		\draw [thin=1.00] (14) to (37);
		\draw [thin=1.00] (37) to (60);

		\draw [thin=1.00] (15) to (38);
		\draw [thin=1.00] (38) to (61);

		\draw [thin=1.00] (16) to (39);
		\draw [thin=1.00] (39) to (62);

	  \draw [very thick=1.00] (17) to (40);
		\draw [thin=1.00] (40) to (63);

		\draw [thin=1.00] (18) to (41);
		\draw [thin=1.00] (41) to (64);

		\draw [thin=1.00] (19) to (42);
		\draw [thin=1.00] (42) to (65);

		\draw [very thick=1.00] (20) to (43);
		\draw [thin=1.00] (43) to (66);

		\draw [thin=1.00] (21) to (44);
		\draw [thin=1.00] (44) to (67);

		\draw [thin=1.00] (22) to (45);
		\draw [thin=1.00] (45) to (68);

		\draw [very thick=1.00] (23) to (46);
		\draw [thin=1.00] (46) to (69);

\end{pgfonlayer}
\end{tikzpicture}
\end{center}
}}

{\tiny{
\begin{center}
\pgfdeclarelayer{nodelayer}
\pgfdeclarelayer{edgelayer}
\pgfsetlayers{nodelayer,edgelayer}
\begin{tikzpicture}
	\begin{pgfonlayer}{nodelayer}
	
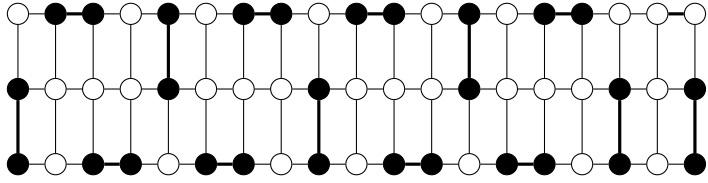
	\node [minimum size=0cm,]  at (-7.5,6.5) {Figure 5. A $G \equiv G_{3,14}$ Grid with $MIM_G=14$};
	
	  \node [minimum size=0cm,draw,fill=black,circle] (1) at (-12,7) {};
		\node [minimum size=0cm,draw,circle] (2) at (-11.5,7) {};
		\node [minimum size=0cm,draw,fill=black,circle] (3) at (-11,7) {};
		\node [minimum size=0cm,draw,fill=black,circle] (4) at (-10.5,7) {};
		\node [minimum size=0cm,draw,circle] (5) at (-10,7) {};
		\node [minimum size=0cm,draw,fill=black,circle] (6) at (-9.5,7) {};
		\node [minimum size=0cm,draw,fill=black,circle] (7) at (-9,7) {};
		\node [minimum size=0cm,draw,circle] (8) at (-8.5,7) {};
		\node [minimum size=0cm,draw,fill=black,circle] (9) at (-8,7) {};
		\node [minimum size=0cm,draw,circle] (10) at (-7.5,7) {};
    \node [minimum size=0cm,draw,fill=black,circle] (11) at (-7,7) {};
		\node [minimum size=0cm,draw,fill=black,circle] (12) at (-6.5,7) {};
		\node [minimum size=0cm,draw,circle] (13) at (-6,7) {};
		\node [minimum size=0cm,draw,fill=black,circle] (14) at (-5.5,7) {};
		\node [minimum size=0cm,draw,fill=black,circle] (15) at (-5,7) {};
		\node [minimum size=0cm,draw,circle] (16) at (-4.5,7) {};
		\node [minimum size=0cm,draw,fill=black,circle] (17) at (-4,7) {};
		\node [minimum size=0cm,draw,circle] (18) at (-3.5,7) {};
		\node [minimum size=0cm,draw,fill=black,circle] (19) at (-3,7) {};

		\node [minimum size=0cm,draw,fill=black,circle] (24) at (-12,8) {};
		\node [minimum size=0cm,draw,circle] (25) at (-11.5,8) {};
		\node [minimum size=0cm,draw,circle] (26) at (-11,8) {};
		\node [minimum size=0cm,draw,circle] (27) at (-10.5,8) {};
		\node [minimum size=0cm,draw,fill=black,circle] (28) at (-10,8) {};
		\node [minimum size=0cm,draw,circle] (29) at (-9.5,8) {};
		\node [minimum size=0cm,draw,circle] (30) at (-9,8) {};
		\node [minimum size=0cm,draw,circle] (31) at (-8.5,8) {};
		\node [minimum size=0cm,draw,fill=black,circle] (32) at (-8,8) {};
		\node [minimum size=0cm,draw,circle] (33) at (-7.5,8) {};
		\node [minimum size=0cm,draw,circle] (34) at (-7,8) {};
		\node [minimum size=0cm,draw,circle] (35) at (-6.5,8) {};
		\node [minimum size=0cm,draw,fill=black,circle] (36) at (-6,8) {};
		\node [minimum size=0cm,draw,circle] (37) at (-5.5,8) {};
		\node [minimum size=0cm,draw,circle] (38) at (-5,8) {};
		\node [minimum size=0cm,draw,circle] (39) at (-4.5,8) {};
		\node [minimum size=0cm,draw,fill=black,circle] (40) at (-4,8) {};
		\node [minimum size=0cm,draw,circle] (41) at (-3.5,8) {};
		\node [minimum size=0cm,draw,fill=black,circle] (42) at (-3,8) {};

		\node [minimum size=0cm,draw,circle] (47) at (-12,9) {};
		\node [minimum size=0cm,draw,fill=black,circle] (48) at (-11.5,9) {};
		\node [minimum size=0cm,draw,fill=black,circle] (49) at (-11,9) {};
		\node [minimum size=0cm,draw,circle] (50) at (-10.5,9) {};
		\node [minimum size=0cm,draw,fill=black,circle] (51) at (-10,9) {};
		\node [minimum size=0cm,draw,circle] (52) at (-9.5,9) {};
		\node [minimum size=0cm,draw,fill=black,circle] (53) at (-9,9) {};
		\node [minimum size=0cm,draw,fill=black,circle] (54) at (-8.5,9) {};
		\node [minimum size=0cm,draw,circle] (55) at (-8,9) {};
		\node [minimum size=0cm,draw,fill=black,circle] (56) at (-7.5,9) {};

		\node [minimum size=0cm,draw,fill=black,circle] (57) at (-7,9) {};
		\node [minimum size=0cm,draw,circle] (58) at (-6.5,9) {};
		\node [minimum size=0cm,draw,fill=black,circle] (59) at (-6,9) {};
		\node [minimum size=0cm,draw,circle] (60) at (-5.5,9) {};
		\node [minimum size=0cm,draw,fill=black,circle] (61) at (-5,9) {};
		\node [minimum size=0cm,draw,fill=black,circle] (62) at (-4.5,9) {};
		\node [minimum size=0cm,draw,circle] (63) at (-4,9) {};
		\node [minimum size=0cm,draw,circle] (64) at (-3.5,9) {};
		\node [minimum size=0cm,draw,circle] (65) at (-3,9) {};

		\end{pgfonlayer}
		\begin{pgfonlayer}{edgelayer}
		\draw [thin=1.00] (1) to (2);
		\draw [thin=1.00] (2) to (3);
		\draw [very thick=1.00] (3) to (4);
		\draw [thin=1.00] (4) to (5);
		\draw [thin=1.00] (5) to (6);
		\draw [very thick=1.00] (6) to (7);
		\draw [thin=1.00] (7) to (8);
		\draw [thin=1.00] (8) to (9);
		\draw [thin=1.00] (9) to (10);
		\draw [thin=1.00] (10) to (11);
		\draw [very thick=1.00] (11) to (12);
		\draw [thin=1.00] (12) to (13);
		\draw [thin=1.00] (13) to (14);
		\draw [very thick=1.00] (14) to (15);
		\draw [thin=1.00] (15) to (16);
		\draw [thin=1.00] (16) to (17);
		\draw [thin=1.00] (17) to (18);
		\draw [thin=1.00] (18) to (19);

			\draw [thin=1.00] (24) to (25);
		\draw [thin=1.00] (25) to (26);
		\draw [thin=1.00] (26) to (27);
		\draw [thin=1.00] (27) to (28);
		\draw [thin=1.00] (28) to (29);
		\draw [thin=1.00] (29) to (30);
		\draw [thin=1.00] (30) to (31);
		\draw [thin=1.00] (31) to (32);
		\draw [thin=1.00] (32) to (33);
			\draw [thin=1.00] (33) to (34);
		\draw [thin=1.00] (34) to (35);
		\draw [thin=1.00] (35) to (36);
		\draw [thin=1.00] (36) to (37);
		\draw [thin=1.00] (37) to (38);
		\draw [thin=1.00] (38) to (39);
		\draw [thin=1.00] (39) to (40);
		\draw [thin=1.00] (40) to (41);
		\draw [thin=1.00] (41) to (42);

			\draw [thin=1.00] (47) to (48);
		\draw [very thick=1.00] (48) to (49);
		\draw [thin=1.00] (49) to (50);
		\draw [thin=1.00] (50) to (51);
		\draw [thin=1.00] (51) to (52);
		\draw [thin=1.00] (52) to (53);
		\draw [very thick=1.00] (53) to (54);
		\draw [thin=1.00] (54) to (55);
		\draw [thin=1.00] (55) to (56);
		\draw [very thick=1.00] (56) to (57);
		\draw [thin=1.00] (57) to (58);
		\draw [thin=1.00] (58) to (59);
		\draw [thin=1.00] (59) to (60);
		\draw [thin=1.00] (60) to (61);
		\draw [very thick=1.00] (61) to (62);
		\draw [thin=1.00] (62) to (63);
		\draw [thin=1.00] (63) to (64);
		\draw [very thick=1.00] (64) to (65);

		\draw [very thick=1.00] (1) to (24);
		\draw [thin=1.00] (24) to (47);

		\draw [thin=1.00] (2) to (25);
		\draw [thin=1.00] (25) to (48);

		\draw [thin=1.00] (3) to (26);
		\draw [thin=1.00] (26) to (49);

		\draw [thin=1.00] (4) to (27);
		\draw [thin=1.00] (27) to (50);

		\draw [thin=1.00] (5) to (28);
		\draw [very thick=1.00] (28) to (51);

		\draw [thin=1.00] (6) to (29);
		\draw [thin=1.00] (29) to (52);

		\draw [thin=1.00] (7) to (30);
		\draw [thin=1.00] (30) to (53);

		\draw [thin=1.00] (8) to (31);
		\draw [thin=1.00] (31) to (54);

			\draw [very thick=1.00] (9) to (32);
		\draw [thin=1.00] (32) to (55);

		\draw [thin=1.00] (10) to (33);
		\draw [thin=1.00] (33) to (56);

		\draw [thin=1.00] (11) to (34);
		\draw [thin=1.00] (34) to (57);

		\draw [thin=1.00] (12) to (35);
		\draw [thin=1.00] (35) to (58);

		\draw [thin=1.00] (13) to (36);
		\draw [very thick=1.00] (36) to (59);

		\draw [thin=1.00] (14) to (37);
		\draw [thin=1.00] (37) to (60);

		\draw [thin=1.00] (15) to (38);
		\draw [thin=1.00] (38) to (61);

		\draw [thin=1.00] (16) to (39);
		\draw [thin=1.00] (39) to (62);

	  \draw [very thick=1.00] (17) to (40);
		\draw [thin=1.00] (40) to (63);

		\draw [thin=1.00] (18) to (41);
		\draw [thin=1.00] (41) to (64);

		\draw [very thick=1.00] (19) to (42);
		\draw [thin=1.00] (42) to (65);

\end{pgfonlayer}
\end{tikzpicture}
\end{center}
}}

\begin{thm}\label{thm5} Let $G$ be a $G_{n,m}$ grid, with $m \geq 23$. Then for $n \equiv 1 \;mod\; 4$, $MIM_G \leq \left\lfloor \frac{2mn-m-3}{8}\right\rfloor$.
\end{thm}

\begin{proof} For $n \equiv 1 \mod 4$, $n-5 \equiv 0 \mod 4$. Let $G_1$ and $G_2$ be partitions of $G$ with $G_1$ induced by $\left\{U_i,U_2, \cdots, V_{n-5}\right\}$ and $\left\{V_{n-4}, V_{n-3}, V_{n-2}, V_{n-1} V_n\right\}$ respectively. Also, let $M',M''$ be $MIM$ of $G_1$ and $G_2$ respectively.
Suppose, $U_{n-5}$ contains at least $\frac{m-1}{2}$ saturated vertices, the least $U_{n-5}$ can contain for $M'$ to remain $MIM$ of $G_1$.
By Theorem \ref{thm1}, $U_1 \subset G_2$ (the $U_{n-4}$ of $G$) contains at least $2k+2$ saturated vertices with $k=\frac{m-3}{4}$. Following the proof of Theorem \ref{thm1}, it is shown that $M''$ contains $\frac{m-3}{4}$ edges of $U_1 \subset G_2$ and either of $u_{(^1_2)}v_4$ and $u_{(^1_2)}v_{m-3}$. Now, with $G=G' \cup G''$ and hence, $|M| \leq |M'|+|M''|$, it is obvious therefore, that for each edge $u_{\alpha}u_{\beta} \in U_{n-4}$ contained in $M''$, either $u_{\alpha}$ or $u_{\beta}$ is adjacent to a saturated vertex on $U_{n-5}$ and also, $u_{n-4}v_4$ (or $u_{n-4}v_{m-3}$) is is adjacent to saturated $u_{n-5}v_4$ (or to saturated $u_{n-4}v_{m-3}$). Hence, $|V_{st}(G)| \leq \frac{2mn-m-7}{4}$ and thus, $MIM_G \leq \left\lfloor \frac{2mn-m-7}{8}\right\rfloor$.
\end{proof}

\begin{thm} Let $G$ be a $G_{n,m}$ grid with $n\equiv 3 \mod 4$ and $m \equiv 3 \mod 4$, $m \geq 11$. Then $MIM_G  \leq \left\lfloor \frac{2mn-m+1}{8}\right\rfloor$ and  $MIM_G  \leq \left\lfloor \frac{2mn-m+5}{8}\right\rfloor$ for $m=8k'+3$ and $m=8k'+7$ respectively.
\end{thm}

\begin{proof} The proof follows similar technique as in Theorem \ref{thm5}.
\end{proof}

\end{document}